\documentclass[10pt]{amsart}

\usepackage{amsfonts,amssymb,amsmath,amscd,amstext}
\usepackage[colorlinks=true,linkcolor=blue,citecolor=blue]{hyperref}
\usepackage[utf8]{inputenc}
\usepackage{microtype}
\usepackage{graphicx}
\usepackage{changes}
\usepackage{verbatim}

\renewcommand{\le}{\leqslant}
\renewcommand{\ge}{\geqslant}
\newcommand{\ptl}{\partial}

\newcommand{\rr}{{\mathbb{R}}}
\newcommand{\cc}{{\mathbb{C}}}

\newcommand{\hh}{{\mathbb{H}}}
\newcommand{\nn}{{\mathbb{N}}}

\newcommand{\eps}{\varepsilon}

\newcommand{\ga}{\gamma}
\newcommand{\Ga}{\Gamma}

\newcommand{\escpr}[1]{\langle#1\rangle}

\newcommand{\mh}{\mathcal{H}}

\DeclareMathOperator{\spn}{span}

\newtheorem{theorem}{Theorem}[section]
\newtheorem{proposition}[theorem]{Proposition}
\newtheorem{lemma}[theorem]{Lemma}
\newtheorem{corollary}[theorem]{Corollary}

\theoremstyle{definition}

\newtheorem{remark}[theorem]{Remark}
\newtheorem{example}[theorem]{Example}

\newtheorem{definition}[theorem]{Definition} 
\theoremstyle{remark}

\numberwithin{equation}{section}

\definechangesauthor[name=Gianmarco, color=blue]{G}
\definechangesauthor[name=Manuel, color=purple]{M}
\definechangesauthor[name=Giovanna, color=violet]{C}

\begin{document}

\title{Variational formulas for curves of~fixed~degree}

\author[G.~Citti]{Giovanna Citti}
\address{Dipartimento di Matematica, Università di Bologna, Piazza di Porta S. Donato 5, 401
26 Bologna, Italy}
\email{giovanna.citti@unibo.it}

\author[G.~Giovannardi]{Gianmarco Giovannardi}
\address{Dipartimento di Matematica, Università degli Studi di Trento, Via Sommarive 14, `
38123, Povo (Trento), Italy}
\email{g.giovannardi@unitn.it}

\author[M.~Ritor\'e]{Manuel Ritor\'e} 
\address{Departamento de Geometr\'{\i}a y Topolog\'{\i}a \\
Universidad de Granada \\ E--18071 Granada \\ Espa\~na}
\email{ritore@ugr.es}

\date{\today}

\thanks{The authors have been supported by Horizon 2020 Project ref. 777822: GHAIA,  MEC-Feder grant MTM2017-84851-C2-1-P (GG and MR) and PRIN 2015 ”Variational and perturbative aspects of nonlinear differential problems” (GC and GG)}
\subjclass[2000]{53C17, 49Q20}
\keywords{Sub-Riemannian geometry, graded manifolds, admissible variations, regular curves, holonomy map}

\thispagestyle{empty}

\bibliographystyle{abbrv} 
\maketitle

\begin{abstract}
We consider a length functional for $C^1$ curves of fixed degree in graded manifolds equipped with a Riemannian metric. The first variation of this length functional can be computed only if the curve can be deformed in a suitable sense, and this condition is expressed via a differential equation along the curve. In the classical differential geometry setting, the analogous condition was considered by Bryant and Hsu in \cite{MR1240644,MR1189496}, who proved that it is equivalent to the surjectivity of a holonomy map. The purpose of this paper is to extend this deformation theory to curves of fixed degree providing several examples and applications. In particular, we give a useful sufficient condition to guarantee the possibility of deforming a curve.
\end{abstract}

\tableofcontents
\section{Introduction}

\thispagestyle{empty}

We consider in this work equiregular graded manifolds $(N,\mh^1,\ldots,\mh^s)$, where $\mh^1\subset\mh^2\subset\cdots\subset\mh^s=TN$ is a flag of sub-bundles of the tangent bundle satisfying $[\mh^i,\mh^j]\subset \mh^{i+j}$, $i,j\ge 1$. This flag allows to define the degree of a tangent vector. 

Examples of graded manifolds are Carnot manifolds $(N,\mh)$, where $\mh$ is a constant rank distribution satisfying H\"ormander's rank condition: in this case $\mh^1$ coincides with $\mh$ and for every $i$, $\mh^i$ is obtained from $\mh$ via $i$ commutations.

Another example are H\"ormander structures of type II introduced by Rothschild and Stein \cite{RothStein}. These structures are naturally associated to a heat subelliptic equation: $\mh^1$ is a purely spatial H\"ormander distribution, $\partial_t\in \mh^2$ and all the elements of the flag of higher degree are obtained via commutation. In this case it is clear that regularity properties of the solution depend not only on integral curves of vector fields of degree one, but also on integral curves of the vector field $\partial_t$, of degree 2. 
Finally we can consider regular submanifolds $N$ of a sub-Riemannian manifold $M$. In \cite[page 151]{Gromov} Gromov points out that, while the distance of a sub-Riemannian manifold $M$ can be expressed in term of integral curves of vector fields of degree 1  by Chow's Theorem, the same thing is no more true for a submanifold $N$ immersed in $M$, with the induced distance. Only if the new distribution $\mh \cap TN$ verifies a H\"ormander type condition on $N$, there exists a horizontal path tangent to $N$ connecting any two points in $N$. This condition for the distribution $\mh \cap TN$ is not verified even in simple cases. Nevertheless the submanifold $N$ inherits a filtration of its tangent bundle $TN$ by means of the flag of sub-bundles in the ambient space $N$ induced by the distribution $\mh$. Therefore $N$ endowed with this induced flag is a graded manifold and the induced anisotropic distance on the submanifold $N$ can be defined as in Definition 1.1 in the paper by Nagel, Stein and Wainger \cite{NagelSteinWainger}. As Gromov suggests in \cite[page 152]{Gromov}, studying integral curves of vector fields of degree greater than one is a crucial tool in understanding the induced distance on the submanifold $N$, when the distribution $\mh \cap TN$ does not verify a H\"ormander type condition.

In sub-Riemannian geometry, the existence of minimizing curves for the length functional that are not solutions of the geodesic equation was discovered by Montgomery in \cite{Mont94a,Mont94b}. These curves are known as abnormal extremals. The problem of their regularity has been widely considered  in the literature, see for instance \cite{Montgomery, AgrachevSachkov, AgrachevBarilariBoscain, LeonardiMonti, Extremalcurves, Monti, AgracevSarychev,Rifford}. The usual approach to face this problem is by means of the study of the endpoint map. However, in this paper we follow an alternative approach based on the Griffiths formalism as suggested by Bryant and Hsu \cite{MR1240644,MR1189496}.

In this paper we consider a general graded manifold $N$, we call degree of a $C^1$ curve $\ga:I\to N$ as the maximum of the pointwise degree of $\ga'(t)$ at every $t\in I$. For these curves we will introduce the notion of length as follows. If we set a Riemannian metric $g=\escpr{\cdot,\cdot}$ on $N$ we get, for any $p\in N$, an orthogonal decomposition $T_pN=\mathcal{K}_p^1\oplus\ldots\oplus\mathcal{K}_p^s$ and, for any $r>0$, the Riemannian metric $g_r$ making the subspaces $\mathcal{K}_p^i$ orthogonal and such that
\[
g_r\big|_{\mathcal{K}^i}=\frac{1}{r^{i-1}}\,g\big|_{\mathcal{K}^i} \; .
\]
These metrics allow to define the length $L_d$ of a curve $\ga$ of degree $d$ in $N$ by 
\[
L_d(\ga)=\lim_{r\downarrow 0} r^{\tfrac{1}{2}(d-1)} L(\ga,g_r),
\]
where $L(\ga,g_r)$ is the Riemannian length of $\ga$  with respect to the metric $g_r$. We are interested in computing the Euler-Lagrange equations for the length functional $L_d(\ga)$. However, we only consider \emph{admissible} variations: the ones that  preserve the degree of the initial curve $\ga$. Then it turns out that the associated variational vector field $V(t)= \frac{\partial \Gamma_s(t)}{\partial s }\big|_{s=0}$ to an admissible variation  has to verify the first order condition \eqref{System of PDEs for admissible} along $\ga$. We say  that a vector field along $\ga$ is admissible when it verifies the system of ODEs \eqref{System of PDEs for admissible}. In \cite[Theorem 3]{MR1189496} Hsu pointed out that under a \emph{regularity} condition on $\ga$ each admissible vector field along $\ga$ is integrable by an admissible variation.

Roughly speaking a curve $\ga$ is regular if it admits enough compactly supported variations preserving its degree. Indeed, to integrate the vector field $V(t)$ we follow the exponential map generating the non-admissible compactly supported variation $\Gamma_s(t)=\exp_{\ga(t)}(s V(t))$ of the initial curve $\ga$. Let $\text{supp}(V)\subset [a,b]$. By the Implicit Function Theorem  there exists a vector field $Y(s,t)$ along $\ga$ vanishing at $a$ such that the perturbations  $\tilde{\Gamma}_s(t)=\exp_{\ga(t)}(s V(t)+ Y(s,t))$ of $\Gamma$ are curves of the same degree of $\ga$ for each $s$ small enough. In general $\tilde{\Gamma}$ fixes the endpoint at $\ga(a)$ but moves the endpoint at $\ga(b)$. Finally the regularity condition on $\ga$ allows us to produce the admissible variation that fixes the endpoint $\gamma(b)$ and integrate $V$.

This concept of regularity deals with the controllability (see \cite[Chapter 13]{Brockett}) of the system of ODEs \eqref{System of PDEs for admissible}. Indeed after splitting the admissible vector $V$ along $\ga$ in its horizontal $V_h=\sum_{i=1}^k g_i X_i$ and vertical $V_v=\sum_{j=k+1}^n f_j X_j $ part the admissibility system of ODEs  \eqref{System of PDEs for admissible} is equivalent to 
\begin{equation}
\label{eq:intro:adm}
F'(t)+B(t)F(t)+A(t)G(t)=0,
\end{equation}
where $A(t), B(t)$ are defined in \eqref{def:AB} and $(X_i)$ is a global orthonormal adapted basis along $\ga$. 
We control this linear system with initial condition $F(a)=0$ on a compact interval $[a,b]\subset I$ when for each value $y_0 \in \rr^{n-k}$ there exists a control horizontal vector field $G(t) \in C_0^{r-1}((a,b), \rr^k)$ such that $F(t)$ solves \eqref{eq:intro:adm} and $F(b)=y_0$. In other words if the \emph{holonomy} map
\[
H_\gamma^{a,b}:\mathcal{H}^{r-1}_0((a,b))\to \mathcal{V}_{\ga(b)}, \qquad  H_{\ga}^{a,b}(G):=F(b)
\]
is surjective the system \eqref{eq:intro:adm} is controllable. Therefore a curve $\ga$ is said to be regular restricted to $[a,b]$ when the holonomy map is surjective. It turns out that  there exists a regular matrix $D(t)$ along $\ga$ solving the differential equation $D'=DB$ such that the image of the holonomy map is given by
\[
H_{\ga}^{a,b}(G)= - D(b)^{-1} \int_a^b D(t) A(t) G(t) dt.
\]
In Corollary 5 in \cite{MR1189496} (Proposition~\ref{prop:linearlyfull}) Hsu proved that the regularity condition on $\ga$ is equivalent to maximal rank condition on the matrix $\tilde{A}(t)= D(t)A(t)$ along $\ga$. Furthermore all singular curves are characterized by the existence of a non-vanishing row vector $\Lambda(t)$ along $\ga$ solving 
\begin{equation}
 \begin{cases}
 \Lambda'(t)=\Lambda(t) B(t)\\
 \Lambda(t) A(t)=0.
 \end{cases}
\end{equation}
Moreover, we check that the surjectivity of the holonomy map is independent of the choice of Riemannian metric $g=\escpr{\cdot,\cdot}$ on the tangent bundle $TN$, thus the regularity of a curve $\ga$ is an invariant  of the graded structure $(N,\mh^1,\ldots,\mh^s)$.

On the other hand the minimizing paths for the length functional $L_d$ of fixed degree clearly strongly depend on the Riemannian metric $g$. Regarding only regular curves of fixed degree we deduce the Euler-Lagrange equation for the critical points of $L_d$ in Theroem~\ref{thm:geod}. However there are singular curves that are not solution of the geodesic equation but they are minima for $L_d$. When this singular horizontal curves are minima of the sub-Riemannian length $L$, they are known as abnormal extremals.

The paper is organized as follows. In Section~\ref{sc:prel} we shortly recall the definitions of graded manifolds and Carnot manifolds, the degree of an immersed curve and the length functional $L_d$ for curves of degree less than or equal to $d$. In Section~\ref{sc:admvariations} we deal with admissible variations for curves of degree $d$ and we deduce the system of ODEs for admissible vector fields. In Section~\ref{sc:structure} the invariances of this system are studied. Section~\ref{sc:holonomy} is completely devoted to description of the holonomy map and characterization of regular and singular curves. Here explicit examples of singular curves of degree greater that one are showed. 
In Section~\ref{sc:weakint} we give in Definition~\ref{def:stronglyregular} a weaker pointwise sufficient condition to ensure the regularity of a curve of degree $d$. This condition does not require solving a differential equation but still ensures the regularity of the curve, see Theorem~\ref{thm:intcriterion2}. This condition \emph{can be easily generalized to submanifolds} of given degree, providing a condition for the regularity of submanifolds in a Carnot manifold \cite{CGR0}. Section~\ref{sc:fvf} is dedicated to the first variation formula for the length functional $L_d$. Some applications to the computation of geodesics in graded manifolds are provided. Finally, in Appendix~\ref{sec:integrability of admissible vector} we provide a different proof of Theorem 3 in  \cite{MR1189496} using the Implicit Function Theorem in Banach spaces  and in Appendix~\ref{ap:holonomyL2} we extend the holonomy map to curves with square integrable derivatives. We include these sections for reader's convenience.

\section{Preliminaries}
\label{sc:prel}

Let $N$ be an $n$-dimensional smooth manifold. Given two smooth vector fields $X,Y$ on $N$, their \emph{commutator} or \emph{Lie bracket} is defined by $[X,Y]:=XY-YX$.  An \emph{increasing filtration} $(\mh^i)_{i\in \nn}$ of  the tangent bundle $TN$ is a flag of sub-bundles
\begin{equation}
\mathcal{H}^1\subset\mathcal{H}^2\subset\cdots\subset \mathcal{H}^i\subset\cdots\subseteq TN,
\label{manifold flag}
\end{equation}
such that
\begin{enumerate}
\label{def:incfilt}
\item[(i)] $ \cup_{i \in \nn} \mh^i= TN$ 
\item[(ii)]
 $ [\mathcal{H}^{i},\mathcal{H}^{j}] \subseteq \mathcal{H}^{i+j},$ for $ i,j \ge1$, 
\end{enumerate}
where  
$ [\mathcal{H}^i,\mathcal{H}^j]:=\spn\{[X,Y]  : X \in \mathcal{H}^i,Y \in \mathcal{H}^j\}$.
Moreover, we say that an increasing filtration is \emph{locally finite} when
\begin{enumerate}
\item[(iii)]  for each $p \in N$ there exists an integer  $s=s(p)$, the \emph{step} at $p$, satisfying $\mathcal{H}^s_p=T_p N$. Then  we have the following flag of subspaces
\begin{equation}
 \mathcal{H}^1_p\subset\mathcal{H}^2_p\subset\cdots\subset \mathcal{H}^s_p=T_p N.
 \label{flag at each point}
\end{equation}
\end{enumerate}

A \textit{graded manifold} $(N,(\mh^i))$ is a smooth manifold $N$ endowed
with a locally finite increasing filtration, namely  a flag of sub-bundles \eqref{manifold flag} satisfying (i),(ii) and (iii). For the sake of brevity a locally finite increasing filtration will be simply called a filtration.
Setting $n_i(p):=\dim {\mathcal{H}}^i_p $, the integer list $(n_1(p),\cdots,n_s(p))$  is called the \textit{growth vector} of the filtration \eqref{manifold flag} at $p$. When the growth vector is constant in a neighborhood of a point $p \in N$ we say that $p$ is a \textit{regular point} for the filtration. We say that a filtration $(\mathcal{H}^i)$ on a manifold $N$ is \textit{equiregular} if the growth vector is constant in $N$. From now on we suppose that $N$ is an equiregular graded manifold.

Given a vector $v$ in $T_p N$ we say that the \textit{degree} of $v$ is equal to $\ell$ if $v\in\mathcal{H}_p^\ell$ and $v \notin \mathcal{H}_p^{\ell-1}$. In this case we write $\text{deg}(v)=\ell$. The degree of a vector field is defined pointwise and can take different values at different points.

Let $(N,(\mh^1,\ldots, \mh^s))$ be an equiregular graded manifold. Take $p\in N$ and consider an open neighborhood $U$ of $p$  where a local frame $\{X_{1},\cdots,X_{n_1}\}$  generating  $\mathcal{H}^1$ is defined. Clearly the degree of $X_j$, for $j=1,\ldots,n_1$, is equal to one since the vector fields $X_1,\ldots,X_{n_1}$ belong to $\mh^1$. Moreover the vector fields $X_1, \ldots,X_{n_1}$ also lie in $\mh^2$. Setting $n_0:=0$, we add some vector fields $X_{n_{1}+1},\cdots,X_{n_2} \in \mathcal{H}^2\setminus \mathcal{H}^{1} $ so that $(X_1)_p,\ldots,(X_{n_2})_p$ generate $\mathcal{H}^2_p$. Reducing $U$ if necessary we have that $X_1,\ldots,X_{n_2}$ generate $\mathcal{H}^2$ in $U$.  Iterating this procedure we obtain a basis of $TM$ in a neighborhood of $p$
\begin{equation}
\label{local adapted basis to the bundle}
 (X_1,\ldots,X_{n_1},X_{n_1+1},\ldots,X_{n_2},\ldots,X_{n_{s-1}+1}, \ldots,X_n),
\end{equation}
such that the vector fields $X_{n_{i-1}+1},\ldots,X_{n_i}$ have degree equal to $i$. The basis obtained in (\ref{local adapted basis to the bundle}) is called an \textit{adapted basis} to the filtration $(\mh^1,\ldots,\mh^s)$.

\subsection{Degree of a curve}
Let $I$ be a non-trivial interval, and $\ga:I \to N$ a curve of class $C^1$ immersed in an equiregular graded manifold $(N,\mh^1,\ldots,\mh^s)$. Then, following \cite{DonneMagnani, vittonemagnani}, we define the degree of $\ga$ at a point $t\in I$ by 
\[
\deg_{\ga}(t):=\deg(\ga'(t)).
\]
The \textit{degree} $\deg(\ga)$ of a curve $\ga$ is the positive integer 
\[
\deg(\ga):=\max_{t \in I} \deg_{\ga}(t).
\]
We define the \textit{singular set} of $\ga$ by  $\ga_0=\ga(I_0)$ where
\begin{equation}
\label{singular set}
 {I}_0=\{t \in I : \deg_{\ga}(t)<\deg(\ga) \}.
\end{equation}

\begin{proposition}
\label{semicont of the degree}
 Let $\ga: I \to N$ be a $C^1$ immersed curve  in a graded manifold $(N,(\mh^i))$. Let $t_0 \in I$. Assume that $\ga(t_0)$ is a regular point of the filtration. Then we have 
 \[
  \liminf\limits_{t\rightarrow t_0} \deg_{\ga}(t)\geqslant \deg_{\ga}(t_0).
 \]
\end{proposition}
\begin{proof} 
As $p_0=\ga(t_0)\in N$ is regular, there exists a local adapted basis $(X_1,\ldots,X_n)$ in an open neighborhood $U$ of $p_0$. We express the continuous tangent vector  $\ga'$ in  $U$ as
 \begin{equation}
  \ga'(t)=\sum_{i=1}^{s} \sum_{j=n_{i-1}+1}^{n_{i}} h_{ij}(t) (X_j)_{\ga(t)}
  \label{vector written in basis}
\end{equation}
with respect to an  adapted basis $(X_1,\cdots, X_n)$, where $h_{ij}$ are continuous functions on $I$.
Suppose that the degree $\deg(\ga'(t_0))$ is equal to $d \in \nn$. Then, 
there exists an integer $k \in \{ n_{d-1}+1,\cdots,n_{d}\}$  such that
$h_{dk}(t_0)\ne 0$ and $h_{ij}(t_0)=0$ for all $i=d+1,\cdots,s$ and $j=n_{i-1}+1,\cdots,n_{i}$. 
By continuity, there exists an open neighborhood $I' \subset I$ of $t_0$ such that $h_{dk}(t)\ne 0$
for each $t$ in $I'$. Therefore for each $t$ in $I'$ the degree of $\ga'(t)$ is greater than or equal to the degree of $\ga'(t_0)$,
\[
 \deg(\ga'(t))\geqslant \deg(\ga'(t_0))=d.
\]
Taking limits we get
\[
\liminf \limits_{t\rightarrow t_0}\deg_{\ga}(t)\geqslant \deg_{\ga}(t_0). \qedhere
\]
\end{proof}

\begin{corollary}
\label{lsc cor}
Let $\ga: I \to N$ be a $C^1$  immersed curve in an equiregular graded manifold $(N,\mh^1,\ldots,\mh^s)$. Then 
\begin{enumerate}
\item $\deg_{\ga}$ is a lower semicontinuous function on $I$.
\item The singular set $I_0$ defined in \eqref{singular set} is closed in $I$.
\end{enumerate} 
\end{corollary}

\begin{proof}
The first assertion follows from Proposition~\ref{semicont of the degree} since every  point in an equiregular graded manifold is regular. To prove 2, we take $\bar{t} \in I\smallsetminus I_0$. By 1, there exists a open neighborhood $I'$ of $\bar{t}$ in $I$ such that each point $t$ in $I'$ has degree $\deg_{\ga}(t)$ equal to $\deg(\ga)$. Therefore we have $I'\subset I\smallsetminus I_0$ and hence $I\smallsetminus I_0$ is an open set.  
\end{proof}

\subsection{Carnot manifolds}
Let $N$ be an $n$-dimensional smooth manifold. An \emph{$l$-dimensional distribution} $\mathcal{H}$ on $N$ assigns smoothly to every $p\in N$ an $l$-dimensional vector subspace $\mathcal{H}_p$ of $T_pN$. We say that a distribution $\mathcal{H}$ complies \emph{H\"ormander's condition} if  any local frame $\{X_1, \ldots, X_l\}$ spanning $\mathcal{H}$ satisfies
\[
\dim(\mathcal{L}(X_1,\ldots,X_{l}))(p)=n, \quad \text{for all} \ p\in N,
\]
where $\mathcal{L}(X_{1},\ldots,X_{l})$ is the linear span of the vector fields  $X_{1},\ldots,X_{l}$ and their  commutators of any order.

A \emph{Carnot manifold} $(N,\mathcal{H})$ is a smooth manifold $N$ endowed with an $l$-dimensional distribution $\mathcal{H}$ satisfying H\"ormander's condition. We refer to $\mathcal{H}$ as the \emph{horizontal distribution}. We say that a vector field on $N$ is \emph{horizontal} if it is tangent to the horizontal distribution at every point. A $C^1$ path is horizontal if the tangent vector is everywhere tangent to the horizontal distribution. A \emph{sub-Riemannian manifold} $(N,\mathcal{H},h)$ is a Carnot manifold $(N,\mathcal{H})$ endowed with a positive-definite inner product $h$ on $\mathcal{H}$. Such an inner product can always be extended to a Riemannian metric on $N$. Alternatively, any Riemannian metric on $N$ restricted to $\mathcal{H}$ provides a structure of sub-Riemannian manifold. Chow's Theorem assures that in a Carnot manifold $(N,\mathcal{H})$ the set of points that can be connected to a given point $p\in N$ by a horizontal path is the connected component of $N$ containing $p$, see \cite{Montgomery}.
Given a Carnot manifold $(N,\mathcal{H})$, we have a flag of subbundles
\begin{equation}
\mathcal{H}^1:=\mathcal{H}\subset\mathcal{H}^2\subset\cdots\subset \mathcal{H}^i\subset\cdots\subset TN,
\label{eq:c manifold flag}
\end{equation}
defined by 
\begin{equation*}
\mathcal{H}^{i+1} :=\mathcal{H}^i + [\mathcal{H},\mathcal{H}^i], \qquad i\ge 1,
\end{equation*}
where
\begin{equation*}
 [\mathcal{H},\mathcal{H}^i]:=\spn\{[X,Y]  : X \in \mathcal{H},Y \in \mathcal{H}^i\}.
\end{equation*}
The smallest integer $s$ satisfying $\mathcal{H}^s_p=T_pN$ is called the \textit{step} of the distribution $\mathcal{H}$ at the point $p$. Therefore, we have 
\begin{equation}
 \mathcal{H}_p\subset\mathcal{H}^2_p\subset\cdots\subset \mathcal{H}^s_p=T_p N.
 \label{flag at each point}
\end{equation}
 The integer list $(n_1(p),\cdots,n_s(p))$  is called the\textit{ growth vector} of $\mathcal{H}$ at $p$. When the growth vector is constant in a neighborhood of a point $p \in N$ we say that $p$ is a \textit{regular point} for the distribution. We say that a distribution $\mathcal{H}$ on a manifold $N$ is \textit{equiregular}
if the growth vector is constant in $N$. This flag of sub-bundles \eqref{eq:c manifold flag} associated to a Carnot manifold $(N,\mh)$ gives rise to the graded structure $(N,(\mh^i))$. Clearly an equiregular Carnot manifold $(N,\mh)$ of step $s$ is a equiregular graded manifold $(N,\mh^1, \ldots, \mh^s)$.

Given a connected sub-Riemannian manifold $(N,\mathcal{H},h)$, and a $C^1$ horizontal path $\gamma:[a,b]\to N$, we define the length of $\gamma$ by
\begin{equation}
L(\gamma)=\int_a^b \ \sqrt{h(\dot{\gamma}(t),\dot{\gamma}(t))} \ dt.
 \label{length fun}
\end{equation}
By means of the equality
\begin{equation}
 d_c(p,q):=\inf \{L(\gamma) :  \gamma\  \text{is a } C^1\ \text{horizontal path joining } p,q  \in N \},
 \label{C-C distance}
\end{equation}
this length defines a distance function (see \cite[\S~2.1.1,\S~2.1.2]{BuragoYuri}) usually called the \textit{Carnot-Carathéodory distance}, or CC-\emph{distance} for short. See \cite[Chapter 1.4]{Montgomery} for further details.

\subsection{Submanifolds immersed in Carnot manifolds}
\label{sc:submanifold}
Let $N$ be a submanifold immersed in an equiregular Carnot manifold $(M,\mh)$ of step $s$. The intersection subspace $\tilde{\mh}_p:=\mh_p \cap T_p N$ at each point $p \in N$ generates a distribution $\tilde{\mh}$ on $N$. Since a priori the distribution $\tilde{\mh}$ does not satisfy H\"ormander's condition, the structure $(N,\tilde{\mh})$ is not a Carnot manifold. Nevertheless, setting $\tilde{\mh}^i:=TN \cap \mh^i $, the submanifold $N$ inherits a locally finite increasing filtration $\tilde{\mh}^1\subset \ldots \subset \tilde{\mh}^s=TN$, that at each point in $N$ is given by
\begin{equation}
\label{tangent flag of M}
 \tilde{\mathcal{H}}_p^1 \subset \tilde{\mathcal{H}}_p^2 \subset\cdots\subset \tilde{\mathcal{H}}_p^s =T_pN,
\end{equation}
where $\tilde{\mathcal{H}}_p^j=T_pN \cap \mathcal{H}_p^j $ and 
$\tilde{n}_j (p)=\text{dim}(\tilde{\mathcal{H}}_p^j)$. Evidently, (i) defined at the beginning of Section~\ref{def:incfilt} is satisfied. On the other hand, if $X\in\tilde{\mathcal{H}}^i$ and $Y\in\tilde{\mh}^j$, we can extend both vector fields in a neighborhood of $M$ so that the extensions $X_1$, $Y_1$ lie in $\mh^i$ and $\mh^j$, respectively. Then $[X,Y]$ is a tangent vector to $N$ that coincides on $M$ with $[X_1,Y_1]\in \mh^{i+j}$. Hence $[X,Y]\in\tilde{\mh}^{i+j}$. This implies condition (ii) defined at the beginning of Section~\ref{def:incfilt}. 
In \cite[0.6.B]{Gromov} Gromov defines the degree at $p$ by
\begin{equation*}
 \deg_{N} (p)= \sum_{j=1}^s j (\tilde{n}_j (p)- \tilde{n}_{j-1} (p)),
\end{equation*}
setting $\tilde{n}_{0}=0$. The degree $\deg(N)$ of the submanifold $N$ is given by
\[
\deg(N):= \max_{p \in N}  \deg_{N} (p).
\]
We define the \textit{characteristic set} of a submanifold $N$ by 
\begin{equation}
\label{singular set}
 N_0=\{p \in N : \deg_{N}(p)<\deg(N) \}.
\end{equation}
In \cite{CGR0} we proved the set $N\smallsetminus N_0$ is open and the growth vector $(\tilde{n}_1,\ldots, \tilde{n}_s)$ is constant on connected components of $N\smallsetminus N_0$ , then we obtain that $(N\smallsetminus N_0, \tilde{\mh}^1, \ldots,\tilde{\mh}^s)$ is a equiregular graded manifold.

\subsection{Length of a generic curve}
In this section we shall consider an equiregular graded manifold $(N,\mh^1, \ldots, \mh^s)$ endowed with a Riemannian metric $g$. We recall the following construction from \cite[1.4.D]{Gromov}: given $p\in N$, we recursively define the subspaces $\mathcal{K}^1_p:=\mathcal{H}^1_p$, $\mathcal{K}^{i+1}_p:=(\mathcal{H}_p^i)^\perp\cap \mathcal{H}^{i+1}_p$, for $1\le i\le (s-1)$. Here $\perp$ means perpendicular with respect to the Riemannian metric $g$. Therefore we have the decomposition of $T_pN$ into orthogonal subspaces
\begin{equation}
\label{eq:decomp}
T_pN=\mathcal{K}_p^1\oplus \mathcal{K}_p^2\oplus\cdots \oplus\mathcal{K}_p^s.
\end{equation}
Given $r>0$, a unique Riemannian metric $g_r$ is defined under the conditions: (i) the subspaces $\mathcal{K}_i$ are orthogonal, and (ii)
\begin{equation}
\label{metric blow-up} 
g_r|_{\mathcal{K}_i}=\frac{1}{r^{i-1}}g|_{\mathcal{K}_i}, \qquad i=1,\ldots,s.
\end{equation}
It is well-known that when $(N,\mh)$ is a Carnot manifold the Riemannian distances of $(N,g_r)$ uniformly converge to the Carnot-Carathéodory distance of $(N,\mathcal{H},h)$, where $h:=g_{|\mh}$ (see \cite[p.~144]{Gromov}).

Working on a neighborhood $U$ of $p$ we construct  an \emph{orthonormal} adapted basis $(X_1,\ldots,X_n)$ for the Riemannian metric $g$ by choosing orthonormal bases in the orthogonal subspaces $\mathcal{K}^i$, $1\le i\le s$. Let $\ga:I \to N$ be an immersed curve in an equiregular graded manifold $(N,\mh^1,\ldots,\mh^s)$ equipped with the Riemannian metric $g$. By the length formula we get
\begin{equation}
\label{eq:sraintegral}
L(\ga, J , g_r)
=\int_{{J}} \left|\ga'(t) \right|_{g_r} dt,
\end{equation}
where $J\subset I$ is a bounded measurable set on $I$ and $L(\ga, J , g_r)$ is the length of $\ga(J)$ with respect to the Riemannian metric $g_r$. If we set $d=\deg(\ga)$ then we have 
\[
\ga'(t)=   \sum_{j=1}^n h_{j}(t) (X_j)_{\ga(t)},
\]
where $h_j(t)=\escpr{\ga'(t),(X_j)_{\ga(t)}}$, setting $g(\cdot,\cdot)=\escpr{\cdot, \cdot}$. Then it follows 
\[
|\ga'(t)|_{g_r}=\sqrt{ \sum_{j=1}^n r^{-(\deg(X_j)-1)} h_{j}(t)^2}.
\]
By Lebesgue's dominated convergence theorem we obtain 
\begin{equation}
\lim_{r\downarrow 0} \Big(r^{\tfrac{1}{2}(d-1)} L(\ga, J , g_r)\Big)=\int_{J} 
\sqrt{\sum_{j=n_{d-1}+1}^{n_{d}} h_j(t)^2}  \, dt.
\label{eq:integral_formula_Ad}
\end{equation}

\begin{definition}
\label{def:Ld}
If $\ga:I \to N$ is an immersed curve of degree $d$ in a graded manifold $(N,\mathcal{H})$ endowed with a Riemannian metric $g$, the length $L_d$ of degree $d$ is defined by
\[
L_d(\ga,J):=\lim_{r\downarrow 0} \Big(r^{\tfrac{1}{2}(d-1)}L(\ga, J , g_r)\Big),
\]
for any bounded measurable set $J\subset I$.
\end{definition}
Equation \eqref{eq:integral_formula_Ad} provides the integral formula $L_d(\ga,J)=\int_J\theta_d(t)dt$, where
\begin{equation}
\label{eq:thetad}
\theta_d(t)=\bigg( \sum_{j=n_{d-1}+1}^{n_d} \escpr{\ga'(t), (X_j)_{\ga(t)}}^2 \bigg)^{\frac{1}{2}}.
\end{equation}

\begin{remark} 
 Clearly if $\deg(\ga)=1$ and $(N,\mh)$ is a Carnot manifold, the sub-Riemannian metric is given by $h=g|_{\mh}$ then the length functional $L_d$ coincides with the length $L$ defined in \eqref{length fun}.
\end{remark}

\section{Admissible variations for the length functional}
\label{sc:admvariations}
Since the degree is defined by an open condition, the degree can not decrease along a variation $\Gamma(t,s)$ of $\ga(t)$ in a tubular neighborhood of a curve. If it increases strictly, the length functional $L_d$, where $d$ is the degree of the original curve, takes infinite values and we cannot compute the first variation of the functional.

Let us consider a curve $\ga:I \to N$ into an equiregular graded manifold endowed with a Riemannian metric $g=\escpr{\cdot,\cdot}$. In this setting we have the following definition

\begin{definition}
\label{def:admissible}
A smooth map $\Gamma: I \times (-\eps,\eps)\to N$ is said to be \emph{an admissible variation} of $\ga$ if $\Gamma_s: I\to N$, defined by $\Gamma_s(t):=\Gamma(t,s)$, satisfies the following properties
\begin{enumerate}
\item[(i)] $\Gamma_0=\ga$,
\item[(ii)] $\Gamma_s(I)$  is an curve of the same degree as $\ga$ for small enough $s$,
\item[(iii)] $\Gamma_s(t)=\ga(t)$ for $t$ outside a given compact subset of $I$ for all $s \in (-\eps,\eps)$.
\end{enumerate}
\end{definition}

\begin{definition}
Given an admissible variation $\Gamma$, the \emph{associated variational vector field} is defined by
\begin{equation}
\label{eq:admissibleV}
V(t):=\frac{\ptl\Gamma}{\ptl s}(t,0).
\end{equation}
\end{definition}

The vector field $V$ is compactly supported in $I$. We shall denote by $\mathfrak{X}_0(I,N)$ the set of smooth vector fields along $\ga$. Hence $V\in \frak{X}_0(I,N)$ if and only if $V$ is a smooth map $V:I \to TN$ such that $V(t)\in T_{\ga(t)}N$ for all $t \in I$, and is equal to $0$ outside a compact subset of $I$.

Let us see now that the variational vector field $V$ associated to an admissible variation $\Ga$ satisfies a differential equation of first order. Let  $(X_1, \cdots, X_n)$ be an adapted frame in a neighbourhood $U$ of $\ga(t)$ for some $t\in I$. We denote by $d=\deg(\ga)$ the degree of $\ga$. As $\Gamma_s(I)$ is a curve of the same degree as $\ga(I)$ for small $s$, there follows
\begin{equation}
\label{flow that preserves degree curve}
\Big\langle\dfrac{\partial \Gamma(s,t)}{\partial t}, (X_r)_{\Gamma_s(t)} \Big \rangle =0,
\end{equation}
for all $r=n_d + 1, \ldots, n$. Taking derivative with respect to $s$ in equality \eqref{flow that preserves degree curve} and evaluating at $s=0$ we obtain the condition
\begin{equation*}
\escpr{\nabla_{\ga'(t)}V(t), (X_r)_{\ga(t)}}+\escpr{\ga'(t),\nabla_{V(t)} X_r}=0
\end{equation*}
for all $r=n_d + 1, \ldots, n$. In the above formula, $\escpr{\cdot,\cdot}$ indicates the scalar product in $N$. The symbol $\nabla$ denotes, in the left summand, the covariant derivative of vectors in $\frak{X}(I,N)$ induced by $g$ and, in the right summand, the Levi-Civita connection associated to $g$ . Thus, if a variation preserves the degree then the associated variational vector field satisfies the above condition and we are led to the following definition.
\begin{definition}
Given an curve $\ga: I\to N$, a vector field $V\in \mathfrak{X}_0(I,N)$ along $\ga$ is said to be \emph{admissible}  if it satisfies the system of first order PDEs
 \begin{equation}
 \label{System of PDEs for admissible}
\escpr{\nabla_{\ga'(t)}V(t), (X_r)_{\ga(t)}}+\escpr{\ga'(t),\nabla_{V(t)} X_r}=0
\end{equation}
where $r=n_d + 1, \ldots, n$ and $t\in I$. We denote by $\mathcal{A}_{\ga}(I,N)$ the set of admissible vector fields.
\end{definition}

Thus we are led naturally to a problem of integrability: given a vector field $V$ along $\ga$ such that the first order condition \eqref{System of PDEs for admissible} holds, we wish to find an admissible variation whose associated variational vector field is $V$.

\begin{definition}
We say that an admissible vector field $V\in\frak{X}_0(I,N)$ is \emph{integrable} if there exists an admissible variation such that the associated variational vector field is $V$.
\end{definition}

\section{The structure of the admissibility system of ODEs}
\label{sc:structure}

Let $(N,(\mathcal{H}^i))$ be an equiregular graded manifold endowed with a Riemannian metric $g=\escpr{\cdot,\cdot}$. We set $\mh:=\mh^d$, where $1\le d \le s$. For sake of simplicity the distribution $\mathcal{H}$ will be called horizontal as well as a curve of degree $d$  and we set  $k:=n_d$. Let $\gamma:I\to N$ be a horizontal curve defined in an open interval $I\subset\rr$. Take $a<b$ so that $[a,b]\subset I$.

Given an open set $U$ where an orthonormal adapted basis $(X_i)$ is defined, the admissibility condition \eqref{System of PDEs for admissible} for a vector field $V$ is 
\begin{equation}
\label{eq:admissibility-1}
\escpr{\nabla_{\ga'} V,X_r}+\escpr{\ga',\nabla_{V} X_r}=0,\quad r=k+1,\ldots,n.
\end{equation}
Expressing $V$ in terms of $(X_i)$
\[
V=\sum_{i=1}^k g_i X_i+\sum_{j=k+1}^n f_j X_j,
\]
we get that \eqref{eq:admissibility-1} is equivalent to the system of $(n-k)$ first order ordinary differential equations
\begin{equation}
\label{eq:compatibilitycond-2}
f_r'+\sum_{j=k+1}^n b_{rj} f_j+\sum_{i=1}^k a_{ri}g_i= 0,\quad r=k+1,\ldots,n,
\end{equation}
where
\begin{equation}
\label{def:AB}
a_{ri}=\escpr{\nabla_{\ga'}X_i,X_r}+\escpr{\nabla_{X_i}X_r,\ga'},\quad
b_{rj}=\escpr{\nabla_{\ga'}X_j,X_r}+\escpr{\nabla_{X_j}X_r,\ga'}.
\end{equation}

\begin{remark}
\label{rk:conststructure}
Assume that we can extend the tangent vector along $\ga$
\[
\ga'(t)= \sum_{\ell=1}^k h_{\ell}(t) (X_{\ell})_{\ga(t)},
\]
to a vector field on a tabular neighborhood of $\ga$, then we have 
\begin{align*}
a_{ri}&=\escpr{\nabla_{\ga'}X_i,X_r}+\escpr{\nabla_{X_i}X_r,\ga'}\\
         &=\escpr{[\ga',X_i], X_r}+ \escpr{ \nabla_{X_i} \ga', X_r}+\escpr{\nabla_{X_i}X_r,\ga'}\\
         &=\escpr{[\ga',X_i], X_r}+ X_i \escpr{ \ga', X_r}\\
          &=\sum_{\ell=1}^k \escpr{[h_{\ell} X_{\ell},X_i], X_r}=\sum_{\ell=1}^k h_{\ell} \escpr{[ X_{\ell},X_i], X_r}=\sum_{\ell=1}^k h_{\ell} \, c_{\ell i}^r
\end{align*}
and 
\[
b_{rj}=\sum_{\ell=1}^k h_{\ell} \, c_{\ell j}^r
\]
where $c_{\ell i}^r$ and  $c_{\ell j}^r$ for $i,\ell=1,\ldots,k$ and $j,r=k+1,\ldots,n$ are the structure functions, see for instance  \cite{Montgomery}. 

In the special case when $\mh$ is a distribution of a Carnot manifold $(N,\mh)$ the matrix $A(t)=(a_{ir})$ represents the  $\mh_{\ga(t)}$-curvature and $B(t)=(b_{jr})$  the  $H^i$-curvature restricted to $\ga'$ in the first term  with respect to metric $g$, where $H^i=\mh^i_{\ga(t)} / \mh^{i-1}_{\ga(t)}$ for $i=2, \ldots,s$, see for instance \cite{Montgomery,Gromov2, Montefalcone}.
\end{remark}

The system \eqref{eq:compatibilitycond-2} can be written in matrix form as
\begin{equation}
\label{eq:compmatrix}
F'=-BF-AG,
\end{equation}
where $B(t)$ is a square matrix of order $(n-k)$ and $A(t)$ is a matrix of order $(n-k)\times k$, and
\begin{equation}
\label{def:FG}
F=\begin{pmatrix} f_{k+1} \\ \vdots \\ f_n \end{pmatrix}, \quad G=\begin{pmatrix} g_1 \\ \vdots \\ g_k \end{pmatrix}.
\end{equation}

The system \eqref{eq:compmatrix} makes sense for any adapted orthonormal basis $(Y_i)$ defined on the curve $\gamma$, locally extended in a tubular neighborhood of the curve. Indeed, if $(X_i)$ and $(Y_i)$ are two of such adapted bases, we may write
\[
Y_i=\sum_{j=1}^n m_{ij} X_j,
\]
for some square matrix $M=(m_{ij})$ of order $n$. Since $(X_i)$ and $(Y_i)$ are adapted basis, $M$ is a block diagonal matrix
\[
M=\begin{pmatrix}
M_h & 0 \\
0 & M_v
\end{pmatrix},
\]
where $M_h$ and $M_v$ are square matrices of orders $k$ and $(n-k)$, respectively. Let us express $V$ as a linear combination of $Y_i$
\[
V=\sum_{i=1}^k \tilde{g}_i Y_i+\sum_{j=k+1}^n \tilde{f}_j Y_j,
\]
and let $\tilde{A}$, $\tilde{B}$ the associated matrices 
\begin{align*}
\tilde{A}=\big(\escpr{\nabla_{\ga'}Y_i,Y_r}+\escpr{\nabla_{Y_i}Y_r,\ga'}\big)_{i=1,\ldots,k}^{r=k+1,\ldots,n},
\\
\tilde{B}=\big(\escpr{\nabla_{\ga'}Y_j,Y_r}+\escpr{\nabla_{Y_j}Y_r,\ga'}\big)_{j=k+1,\ldots,n}^{r=k+1,\ldots,n}.
\end{align*}
Letting
\[
\tilde{F}=\begin{pmatrix} \tilde{f}_{k+1} \\ \vdots \\ \tilde{f}_n \end{pmatrix}, \quad \tilde{G}=\begin{pmatrix} \tilde{g}_1 \\ \vdots \\ \tilde{g}_k \end{pmatrix},
\]
it is immediate to obtain the following equalities
\begin{equation}
\label{eq:tilde}
\begin{split}
\tilde{F}&=M_v F,
\\
\tilde{G}&=M_h G,
\\
\tilde{A}&=M_v A M_h^t,
\\
\tilde{B}&=M_v(M_v')^t+M_vB M_v^ t.
\end{split}
\end{equation}

\begin{remark}
\label{rk:systemtilde}
We observe that the equations in \eqref{eq:tilde} imply that $\tilde{F}'+\tilde{B}\tilde{F}+\tilde{A}\tilde{G}=0$. To prove this formula it is necessary to take into account that $M_h$ and $M_v$ are orthogonal matrices. We also observe that the ranks of $A(t)$ and $\tilde{A}(t)$ coincide for any $t\in I$.
\end{remark}

\begin{remark}
\label{rk:horizontalconnection}
Given a smooth vector field $X$ and a horizontal vector field $Y$ on $N$, we  define a covariant derivative on the bundle of horizontal vector fields by
\[
\nabla^h_XY=(\nabla_XY)_h,
\]
where $(\cdot)_h$ denotes the orthogonal projection over the horizontal distribution. This covariant derivative defines a parallel transport on any curve in $N$ that preserves the Riemannian product of horizontal vector fields. Then we can extend any horizontal orthonormal basis at a given point of $\ga(I)$ to a horizontal orthonormal basis in $\ga(I)$. A similar connection can be defined on the vertical bundle using the projection over the vertical bundle that allows us to build an orthonormal basis on $\ga(I)$ of the vertical bundle. Hence we are able to produce an adapted global basis on $\ga(I)$ (see for instance \cite{TanYang,Cartan}).
\end{remark}

\section{The holonomy map}
\label{sc:holonomy}
In this section we first recall Hsu's construction of the holonomy map \cite{MR1189496}. Let $(N,(\mathcal{H}^i))$ be an equiregular graded manifold endowed with a Riemannian metric $g=\escpr{\cdot,\cdot}$. We set $\mh:=\mh^d$, where $1\le d \le s$. For sake of simplicity the distribution $\mathcal{H}$ will be called horizontal as well as a curve of degree $d$  and we set  $k:=n_d$. Given a horizontal curve $\ga:I\to N$, with $a\in I$, we consider the following spaces
\begin{enumerate}
\item $\frak{X}^r_\ga(a)$, $r\ge 0$, is the set of $C^r$ vector fields along $\ga$ that vanish at $a$.
\item $\mathcal{H}_\ga^r(a)$, $r\ge 0$, is the set of horizontal $C^r$ vector fields along $\ga$ vanishing at $a$.
\item $\mathcal{V}_\ga^r(a)$, $r\ge 0$, is the set of vertical vector fields of class $C^r$ along $\ga$ vanishing at $a$. By a vertical vector we mean a vector in $\mathcal{H}^\perp$ and  we set $\mathcal{V}_{\ga}= (\mathcal{H}_{\ga})^{\perp}$.
\end{enumerate}

We fix an adapted orthonormal basis $(X_i)$ along $\ga$ extended in a neighborhood of $\ga$. The admissibility condition \eqref{eq:admissibility-1} can be expressed globally on $\ga$ using these global vector fields. We define the \emph{admissibility} operator $\text{Am}:\frak{X}_{\ga}\to \mathcal{V}_{\ga}$ by
\begin{equation}
\label{eq:ad}
\text{Am}(Y)=\sum_{i=k+1}^n \big(\escpr{\nabla_{\gamma'}Y,X_i}+\escpr{\gamma',\nabla_Y X_i}\big)\,X_i.
\end{equation}
Observe that $\text{Am}(Y)=0$ implies that $Y$ is an admissible vector field on $\ga$.

The following result is essential for the construction

\begin{lemma}
\label{lem:existY_hY_v}
Let $\ga:I\to N$ be a horizontal curve in a graded manifold $(N,\mathcal{H})$ endowed with a Riemannian metric. Given $Z\in \frak{X}^{r-1}_\ga(a)$, there exist a unique vertical vector field $Y_v \in \mathcal{V}_\ga^r(a)$  such that $\text{\rm Am}(Z_h+Y_v)=Z_v$.
\end{lemma}

\begin{proof}
We choose a global orthonormal adapted basis $(X_i)$ on $\ga$ and write
\[
Z=\sum_{i=1}^k g_iX_i+\sum_{r=k+1}^n z_rX_r,
\]
The vertical vector field $Y_v$ would be determined by their coordinates $(f_r)$ in the vertical basis $(X_r)$, where $r=k+1, \ldots,n$.

Condition $\text{Am}(Z_h+Y_v)=Z_v$ is then equivalent to the system of $(n-k)$ ordinary differential equations
\begin{equation}
\label{eq:compatibilitycond}
f_r'+\sum_{i=1}^k a_{ri}g_i+\sum_{j=k+1}^n b_{rj} f_j= z_r,\quad r=k+1,\ldots,n,
\end{equation}
where
\begin{equation*}
a_{ri}=\escpr{\nabla_{\ga'}X_i,X_r}+\escpr{\nabla_{X_i}X_r,\ga'},\quad
b_{rj}=\escpr{\nabla_{\ga'}X_j,X_r}+\escpr{\nabla_{X_j}X_r,\ga'}.
\end{equation*}

Given $(g_i)$, the system \eqref{eq:compatibilitycond} admits a unique solution defined in the whole interval $I$ with prescribed initial conditions $f_r(a)=0$, for $r=k+1,\ldots,n$. This concludes the proof.
\end{proof}

Given a horizontal curve $\ga:I\to N$ and $[a,b]\subset I$, Lemma~\ref{lem:existY_hY_v} allows us to define a holonomy  type map
\[
H_\gamma^{a,b}:\mathcal{H}^{r-1}_0((a,b))\to \mathcal{V}_{\ga(b)}
\]
like in Hsu's paper \cite{MR1189496}. Here $\mathcal{H}^{r-1}_0((a,b))$ is the space of horizontal vector fields of class $(r-1)$ with compact support in $(a,b)$. In order to define $H_\gamma^{a,b}$ we consider a horizontal vector $V_h\in\mathcal{H}_0^{r-1}(a,b)$ with compact support in $(a,b)$ and we take the only vector field $V_v\in\mathcal{V}^r_\ga(a)$ such that $\text{Am}(V_h+V_v)=0$ provided by Lemma~\ref{lem:existY_hY_v}. Then we define
\[
H_\gamma^{a,b}(V_h)=V_v(b).
\]

\begin{definition}[\cite{MR1189496}]
In the above conditions, we say that $\gamma$ restricted to $[a,b]$ is \emph{regular} if the holonomy map $H_\gamma^{a,b}$ is surjective.
\end{definition}

Given a horizontal curve $\ga:I\to N$, we choose an orthonormal adapted basis $(X_i)$ along $\gamma$. A horizontal vector field $V_h$ can be expressed in terms of this basis as $V_h=\sum_{i=1}^k g_iX_i$. The unique vertical vector field $V_v$ such that $\text{Am}(V_h+V_v)=0$ can be expressed as $V_v=\sum_{i=k+1}^n f_iX_i$. Defining $F$ and $G$ as in \eqref{def:FG}, condition $\text{Am}(V_h+V_v)=0$ is equivalent to $F'=-BF-AG$, where $A$, $B$ are the matrices defined in \eqref{def:AB}. In these conditions, the coordinates of $H_\ga^{a,b}(V_h)=V_v(b)$ in the basis  $(X_i)$ are given by $F(b)$.

The following result allows the integration of the differential equation \eqref{eq:compmatrix} to explicitly compute the holonomy map.

\begin{proposition}
\label{prop:inthol}
In the above conditions, there exists a square invertible matrix $D(t)$ of order $(n-k)$ such that
\begin{equation}
\label{eq:F(b)}
F(b)=- D(b)^{-1}\int_a^b (D A)(t) G(t) \, dt.
\end{equation}
\end{proposition}

\begin{proof}
Lemma~\ref{lm:det} below allows us to find a invertible matrix $D(t)$ such that $D'= DB$. Then equation $F'=-BF-AG$ is equivalent to $(DF)'=- DAG$. Integrating between $a$ and $b$, taking into account that $F(a)=0$, and multiplying by $D(b)^{-1}$, we obtain \eqref{eq:F(b)}.
\end{proof}

\begin{lemma}
\label{lm:det}
Let $B(t)$ be a continuous family of square matrices on the interval $[a,b]$. Let $D(t)$ be the solution of the Cauchy problem
\[
D'(t)=D(t) B(t)\ \text{on }[a,b],  \quad D(a)=I_d.
\]
Then $\det  D(t)\ne0$ for each $t \in [a,b]$.
\end{lemma}
\begin{proof}
By the Jacobi formula we have 
\[
\dfrac{d (\det D(t))}{dt}=\text{Tr}\left(\text{adj}\, D(t)\, \dfrac{d D(t)}{dt}\right),
\]
where $\text{adj} D$ is the classical adjoint (the transpose of the cofactor matrix) of $D$ and $\text{Tr}$ is the trace operator. Therefore
\begin{equation}
\label{eq:detode}
\dfrac{d \det(D(t))}{dt}=\text{Tr}\left((\text{adj}\, D(t)) D(t) B(t) \right)=\det D(t)\, \text{Tr}(B(t)).
\end{equation}
Since $\det D(a)=1$, the solution for \eqref{eq:detode} is given by 
\[
 \det D(t)=e^{\int_a^t \text{Tr}(B(\tau)) \, d \tau}>0,
\]
for all $t\in [a,b]$. Thus, the matrix $D(t)$ is invertible for each $t \in [a,b]$.
\end{proof}

\begin{definition}
We say that the matrix $\tilde{A}(t):=(DA)(t)$ on $\ga$ defined in Proposition \ref{prop:inthol} is linearly full in $\rr^{n-k}$ if and only if 
\[
\dim \left( \text{span} \left\{\tilde{A}^1(t), \ldots, \tilde{A}^k(t)  \quad \forall t \in [a,b] \right\} \right)=n-k,
\]
where $\tilde{A}^i$ for $i=1,\ldots,k$ are the columns of $\tilde{A}(t)$.
\end{definition}

\begin{proposition}
\label{prop:linearlyfull}
The horizontal curve $\gamma$ restricted to $[a,b]$ is regular if and only if $\tilde{A}(t)$ is linearly full in $\rr^{n-k}$.
\end{proposition}

\begin{proof}
Assume that the holonomy map is not surjective. Then the image of $H_\ga^{a,b}$ is contained in a hyperplane of $\mathcal{V}_{\ga(b)}$ expressed in the coordinates associated to the basis $((X_i)_{\ga(b)})_{i=k+1,\ldots,n}$ as a row vector $\Lambda\neq 0$ with $(n-k)$ coordinates. With the notation of Proposition~\ref{prop:inthol} we have
\[
0=\Lambda F(b)=-\Lambda D(b)^{-1}\int_a^b \tilde{A}(t)G(t)dt= - \int_a^b\Ga \tilde{A}(t)G(t)dt,
\]
where $\Ga=\Lambda D(b)^{-1}\neq 0$. As this formula holds for any $G(t)$, we have $\Ga\tilde{A}(t)=0$ for all $t\in [a,b]$. Hence $\tilde{A}$ is not linearly full as their columns are contained in the hyperplane of $\rr^{n-k}$ determined by $\Ga$.

Conversely, assume that $\tilde{A}$ is not linearly full. Then there exists a row vector with $(n-k)$ coordinates $\Ga\neq 0$ such that $\Ga\tilde{A}(t)=0$ for all $t\in [a,b]$. Then
\[
(\Ga D(b))F(b)=\Ga\int_a^b \tilde{A}(t)G(t)dt=0.
\]
Hence the image of the holonomy map is contained in a hyperplane if $\mathcal{V}_{\ga(b)}$ and $\ga$ is not regular.
\end{proof}

The following result provides a useful criterion of non-regularity

\begin{theorem}
\label{th:singchar}
The horizontal curve $\ga$ is non-regular restricted to $[a,b]$  if and only if there exists a row vector field $\Lambda(t)\ne0$  for all $t \in [a,b]$ that solves the following system
\begin{equation}
 \begin{cases}
 \Lambda'(t)= \Lambda(t) B(t)\\
 \Lambda(t) A(t)=0.
 \end{cases}
 \label{eq:singularsys}
\end{equation}
\end{theorem}
\begin{proof}
 
Assume that $\ga$ is nonregular in $[a,b]$, then by Proposition~\ref{prop:linearlyfull} there exists a row vector $\Gamma \ne 0$ such that 
\[
 \Gamma D(t) A(t)=0 
\] 
for all $t\in [a,b]$, where $D(t)$ solves 
\begin{equation}
\label{eq:homD}
\begin{cases}
D(t)'= D(t) B(t)\\
D(a)=I_{n-k}.
\end{cases}
\end{equation}
Since $\Gamma$ is a constant vector and $D(t)$ is a regular matrix by Lemma~\ref{lm:det} , $\Lambda(t):=\Gamma D(t)$ solves the system \eqref{eq:singularsys} and  $\Lambda(t) \ne 0$ for all $t \in [a,b]$.

Conversely, any solution of the system \eqref{eq:singularsys} is given by 
\[
\Lambda(t)= \Gamma D(t),
\]
where $\Gamma=\Lambda(0)\ne 0$ and $D(t)$ solves the equation \eqref{eq:homD}.
Indeed, let us consider a general solution $\Lambda(t)$ of \eqref{eq:singularsys}. If we set
\[
\Phi(t)=\Lambda(t)-\Gamma D(t),
\]
where $\Gamma=\Lambda(0)\ne 0$ and $D(t)$ solves the equation \eqref{eq:homD}, then we deduce 
\[
\begin{cases}
\Phi(t)'=\Phi(t)B(t)\\
\Phi(0)=0.
\end{cases}
\]
Clearly the unique solution of this system is $\Phi(t)\equiv 0$.
Hence we  conclude that $\Gamma \tilde{A}(t)=0$. Thus $\tilde{A}(t)$  is not fully linear and by Proposition~\ref{prop:linearlyfull}  we are done.
\end{proof}

\begin{remark}
\label{rk:Mchequalsing}
Notice that if we write $\Lambda(t)=(\lambda_{k+1}(t),\ldots,\lambda_n(t))$ by Remark \ref{rk:conststructure} the  equation \eqref{eq:singularsys} is equivalent to 
\begin{equation}
\label{eq:rkchara}
\begin{cases}
\lambda'_r(t)=\sum_{j=k+1}^n \lambda_j  b_{j r}=\sum_{j=k+1}^n \sum_{i=1}^k \lambda_j  c_{i r}^j u_{i}\\
\sum_{j=k+1}^n \lambda_j  a_{j h}= \sum_{j=k+1}^n  \lambda_j  c_{i h}^j u_{i}.
\end{cases}
\end{equation}
When $(N,\mh)$ is a Carnot manifold where $\mh$ is a distribution of rank $l$, then we have that $k=l$ and \eqref{eq:rkchara} is equivalent to the characteristic system in \cite[Lemma 5.2.3]{Montgomery} where the first equations $u_j=0$ for $j=l+1,\ldots,n$ are taken for granted since in Theorem \ref{th:singchar} we assume $\ga$ horizontal. In the characteristic system of \cite[Lemma 5.2.3]{Montgomery} we assume $\ga$ absolutely continuous curves with square integrable derivative while in our construction $\ga$ is at least $C^1$, so that the coefficients of $A$ and $B$ are at least continuous.
\end{remark}

\subsection{Independence on the metric}
Let $g$ and $\tilde{g}$ be two Riemannian metrices on $N$ and $(X_i)$  be orthonormal adapted basis with respect to $g$  and $(Y_i)$ with respect to $\tilde{g}$. Clearly we have 
\[
Y_i=\sum_{j=1}^n m_{ji} X_j,
\]
for some square invertible matrix $M=(m_{ji})_{j=1,\ldots,n}^{i=1,\ldots,n}$ of order $n$. Since $(X_i)$ and $(Y_i)$ are adapted basis, $M$ is a block matrix
\[
M=\begin{pmatrix}
M_h & M_{hv}\\
0 & M_v
\end{pmatrix},
\]
where $M_h$ and $M_v$ are square matrices of orders $k$ and $(n-k)$, respectively, and $M_{hv}$ is a $k \times (n-k)$ matrix.
\begin{remark}
One can easily check that the inverse of $M$ is given by the block matrix 
\[
M^{-1}=\begin{pmatrix}
M_h^{-1} & - M_h^{-1} \, M_{hv} M_{v}^{-1}\\
0 & M_v^{-1}
\end{pmatrix}.
\]
Setting $\tilde{G}=(\tilde{g}(X_i,X_j))_{i,j=1,\ldots,n}$ we have 
\[
\tilde{G}=\begin{pmatrix}
\tilde{G}_h & \tilde{G}_{hv}\\
(\tilde{G}_{hv})^t & \tilde{G}_v
\end{pmatrix}=(M^{-1})^t (M^{-1}).
\]
Thus it follows  
\begin{align*}
\tilde{G}_v&= (M_v^{-1})^t M_v^{-1} + (M_v^{-1})^t M_{hv}^t ( M_h^{-1})^t \, M_h^{-1} M_{hv} M_v^{-1},\\
\tilde{G}_{hv}&=-  (M_h^{-1})^tM_h^{-1} M_{hv} M_v^{-1},\\
\tilde{G}_{h}&= (M_h^{-1})^t  M_h^{-1} .
\end{align*}
\label{rk:MG}
\end{remark}
Let $\tilde{A}$ be the associated matrix 
\begin{align*}
\tilde{A}=\big( \tilde{g} (Y_r , [\ga', Y_i ]) \big)^{i=1,\ldots,k}_{r=k+1,\ldots,n}.
\end{align*}
Then a straightforward computation implies
\begin{equation*}
\begin{aligned}
\tilde{A}=& (M_{hv})^t \left(\tilde{G}_{h} \, C_h \, M_h+  \tilde{G}_{hv} \, A \, M_h + \tilde{G}_{h} \,   M_h ' \right)\\
&+ (M_v)^t \left( (\tilde{G}_{hv})^t \, C_h M_h+ \tilde{G}_v \, A \, M_h+(\tilde{G}_{hv})^t M_v' \right),
\end{aligned}
\end{equation*}
and, by Remark~\ref{rk:MG}, we obtain 
\begin{equation}
\label{eq:tildeA}
\begin{aligned}
\tilde{A}&= (M_{hv})^t \left((M_h^{-1})^t M_h^{-1} \,( C_h \, M_h + M_h ' )-  (M_h^{-1})^t M_h^{-1} M_{hv} M_v^{-1} \, A \, M_h   \right)\\
&\qquad - \left(  M_{hv}^t (M_h^{-1})^t M_h^{-1} \, (C_h M_h + M_v')\right)\\
&\qquad +\left(  M_v^{-1} +  M_{hv}^t ( M_h^{-1})^t \, M_h^{-1} M_{hv} M_v^{-1} \right) \, A \, M_h\\
&= M_v^{-1} \, A \, M_h.
\end{aligned}
\end{equation}

Now let $\tilde{B}$ be the associated matrix 
\begin{align*}
\tilde{B}=\big( \tilde{g} (Y_r , [\ga', Y_j ]) \big)^{j=k+1,\ldots,n}_{r=k+1,\ldots,n}.
\end{align*}
Then it is immediate to obtain the following equality
\begin{align*}
\tilde{B}&= M_{vh}^t \Big(\tilde{G}_h C_h M_{hv}+ \tilde{G}_{hv} A M_{hv}+ \tilde{G}_{h} C_{hv} M_v 
\\ &\quad  + \tilde{G}_{hv} B M_v + \tilde{G}_{h} M_{hv}' +\tilde{G}_{hv} M_v'\Big)\\
            & \quad   + M_v^t\Big( (\tilde{G}_{hv})^t C_h M_{hv}+ \tilde{G}_{v} A M_{hv}+ (\tilde{G}_{hv})^t C_{hv} M_v
\\
&\quad + \tilde{G}_{v} B M_v+ (\tilde{G}_{hv})^t M_{hv}'+ \tilde{G}_{v} M_v' \Big).
\end{align*}
By Remark \ref{rk:MG} we obtain 
\begin{equation}
\label{eq:tildeB}
\tilde{B}=  M_v^{-1} \, A \, M_{hv} + M_v^{-1} B  M_v + M_v^{-1} M_v '.
\end{equation}

\begin{proposition}
\label{prop:indmetric}
Let $\ga:I \to N$ be a curve immersed in a graded manifold $(N,\mh^1, \ldots, \mh^s)$. Let $g$ and $\tilde{g}$ be two Riemannian metrics on $TN$. Then $\ga$ is regular in $[a,b]$ with respect to $g$ if and only if $\ga$ is regular in $[a,b]$ with respect to $\tilde{g}$.
\end{proposition} 
\begin{proof}
Let $(Y_i)$ be an orthonormal adapted basis along $\ga$ with respect to $\tilde{g}$. Without loss of generality we assume that the system \eqref{eq:compmatrix} with respect to metric $\tilde{g}$ is given by 
\[
\tilde{F}'(t)=- \tilde{A}(t) \tilde{G}(t)
\]
and $\tilde{B}=0$.
This is not restrictive since  starting by a generic metric $\tilde{g}_0$, by Proposition \ref{prop:inthol}, there exists a matrix $D(t)$ that transforms the metric $\tilde{g}_0$ in the metric $\tilde{g}$. Let $(X_i)$ be an orthonormal adapted basis along $\ga$ with respect to $g$. Then there exists an invertible block matrix 
\[
M=\begin{pmatrix}
M_h & M_{hv}\\
0 & M_v
\end{pmatrix},
\]
such that 
\[
Y_i=\sum_{j=1}^n m_{ji} X_j.
\]
 Then, by Theorem~\ref{th:singchar} a curve $\ga$ is non-regular with respect to $\tilde{g}$ if and only if $\tilde{\Lambda}(t)= \tilde{\Lambda}\ne0$ is a constant row vector such that $\tilde{\Lambda}\tilde{A}(t)=0$ for each $t \in [a,b]$. By equation \eqref{eq:tildeA}  we obtain  
\[
0=\tilde{\Lambda}\tilde{A}(t)=\tilde{\Lambda} M_v^{-1}(t) \, A(t) \, M_h(t).
\]
Since $M_h$ is invertible, setting $\Lambda(t):=\tilde{\Lambda} M_v^{-1}(t)\ne0$ we deduce $\Lambda(t) A(t)=0$. Moreover by \eqref{eq:tildeB} we have 
\[
0=\tilde{B}=M_v^{-1} \, A  \, M_{hv} + M_v^{-1} \,B \, M_v + M_v^{-1} M_v '.
\]
Multiplying both sides we deduce 
\begin{equation}
\label{eq:tildeder}
0=\tilde{\Lambda} M_v^{-1} \,B \, M_v + \tilde{\Lambda} M_v^{-1} M_v'.
\end{equation}
Differentiating the identity $M_v^{-1} M_v$ it follows 
\[
M_v^{-1} M_v '=- (M_v^{-1})' M_v
\]
Putting this identity in \eqref{eq:tildeder} we have 
\[
\tilde{\Lambda} M_v^{-1} \,B \, M_v = \tilde{\Lambda} (M_v^{-1})' M_v.
\]
Since $\tilde{\Lambda}$ is constant and $M_v$ is invertible we conclude  
$\Lambda'(t)= \Lambda(t) B(t)$. By Theorem~\ref{th:singchar}, this means that $\ga$ is non-regular with respect to $g$.
\end{proof}

Proposition~\ref{prop:indmetric}  shows that the definition of regularity for a curve  $\ga$ is independent of the choice of Riemannian metric $g=\escpr{\cdot,\cdot}$ on the tangent bundle $TN$.

\subsection{Some low-dimensional examples}

\begin{remark}
Let $A$ be the matrix defined in \eqref{def:AB} with respect to an adapted basis $(X_i)$ along $\ga$. Notice that if there exists a point $\bar{t}\in (a,b)$ such that 
\begin{equation}
\label{eq:rankA}
\text{rank}\, A(\bar{t})=n-k,
\end{equation}
for some adapted basis $(X_i)$, then the curve $\ga$ is regular in $[a,b]$. In particular, if we assume that \eqref{eq:rankA} holds for some $\bar{t}\in (a,b)$, then the curve $\ga$ is regular in $[a,b]$. Notice that this condition implies $n-k\le k$, that is $\tfrac{n}{2}\le k$.
\end{remark}

\begin{example}
\label{ex:regularcontact}
Any horizontal curve $\ga:I\to M^{2n+1}$ in a contact sub-Riemannian manifold $(M^{2n+1},\mathcal{H}=\text{ker}(\omega))$, is regular. Let $T$ be the Reeb vector field. We extend the vector field $\ga'$ along $\ga$ to a vector field on $M$.  Given a contact manifold $M$, one can assure the existence of a Riemannian metric $g=\escpr{\cdot,\cdot}$ and an $(1, 1)$-tensor field $J$ so that
\begin{equation}
\label{eq:contactmetric}
\escpr{T, X} = \omega(X), \quad  2\escpr{X, J(Y )} = d\omega(X, Y ), \quad  J^2(X) = -X + \omega(X) T.
\end{equation}
The structure given by $(M, \omega, g, J)$ is called a contact Riemannian manifold, see \cite{BlairRGCS} and  \cite{galliRitore}. In particular the structure $(M ,\mh, g)$ is a Carnot manifold. Then, we  fix an orthonormal adapted basis $(X_1, \ldots,X_{2n},T)$ along $\ga$, where $X_i \in \mh$ for $i=1,\ldots,2n$. Then the admissibility equation for $V=\sum_{i=1}^{2n} f_i X_i+ f_{2n+1} T$ is given by 
\[
f_{2n+1}'(t)= - b \, f_{2n+1}(t) - A \left(\begin{array}{c}
         f_{1}(t)\\
         \vdots\\
         f_{2n}(t)
        \end{array}
 \right),
\]
where $b=\escpr{[\ga',T],T}$ and $A=(a_{1}, \ldots, a_{2n})$ with 
\begin{align*}
a_i&=\escpr{\nabla_{\ga'} X_i, T}+\escpr{\nabla_{X_i}T,\ga'}\\
&=\escpr{[\ga',X_i],T}=\omega([\ga',X_i])\\
&=-d \omega(\ga',X_i)=-2 \escpr{\ga', J(X_i)}= 2\escpr{J(\ga'),X_i}.
\end{align*}
Since $J(\ga'(t))\in \mh_{\ga(t)}$ and $J(\ga'(t))\ne0$ for all $t \in I$  we have $\text{rank}\,A(t)=1$ for all $t \in I$. Hence $\gamma$ is regular in every subinterval of $I$.
\end{example}

Here we show a well-known example of horizontal singular curve, discovered by Sussmann in \cite{Sussman}.
\begin{example}
\label{ex:hsc}
An Engel structure $(E,\mh)$ is $4$-dimensional Carnot manifold where $\mh$ is a two dimensional distribution of step $3$. A representation of the Engel group $\mathbb{E}$, which is the tangent cone to each Engel structure, is given by $\rr^4$ endowed with the distribution $\mh$ generated by 
\[
X_1=\partial_{x_1} \quad \text{and} \quad X_2= \partial_{x_2} + x_ 1\partial_{x_3}+ \dfrac{x_1^2}{2} \partial_{x_4}.
\] 
The second layer is generated by
 $$X_3=[X_1,X_2]= \partial_{x_3}+ x_1 \partial_{x_4}$$
 and the third layer by $X_4=[X_1,X_3]=\partial_{x_4}$. Let $\ga: \rr \to \rr^4$ be the horizontal curve parametrized by $\ga(t)=(0,t,0,0)$ whose tangent vector $\ga'(t)$ is given by $\partial_{x_2}$. We consider the Riemannian metric $g=\escpr{\cdot, \cdot}$ that makes $(X_1,\ldots,X_4)$ an orthonormal basis. By Remark \ref{rk:conststructure} we have $b_{rj}=c_{2 j}^r=0$ for $r,j=3,4$ and $a_{ri}=c_{2 i}^r=0$ for $r=3,4$, $i=1,2$ and $(r,i)\ne(3,1)$. Since $c_{2 1}^3=\escpr{[X_2,X_1], X_3}$ we have  $a_{3 1}=-1$. Let $[a,b]\subset \rr$, $a<b$. Then, by Proposition~\ref{prop:inthol} the holonomy map is given by 
\[
 H_{\ga}^{a,b}(G)=F(b)=\left(
 \begin{array}{c}
 \int_{a}^b g_1(t) dt\\
 0
 \end{array} \right),
\]
for all $g_1, g_2 \in C_0^{\infty}([a,b])$. Therefore we deduce the holonomy map is not surjective, thus $\ga$ is non-regular restricted at each interval $[a,b]$. Clearly we observe that the constant matrix
\[
A(t)=\left(
 \begin{array}{c c}
 -1& 0 \\
  0& 0
 \end{array} \right)
\]
is not linearly full since it generates a one dimensional subspace of $\rr^2$.
\end{example}

Slightly changing the distribution of Example \ref{ex:hsc} we show an example singular curve of degree $2$.
\begin{example}
Let us consider $\rr^5$ endowed with the distribution $\mh$ generated by
\[
X_1=\partial_{x_1} \quad \text{and} \quad X_2=\partial_{x_5}+ x_1\left(\partial_{x_2} +\dfrac{ x_ 1}{2} \partial_{x_3}+ \dfrac{x_1^3}{6} \partial_{x_4}\right).
\]
The second layer is generated by 
\[
X_3=[X_1,X_2]=\partial_{x_2} + x_ 1\partial_{x_3}+ \dfrac{x_1^2}{2} \partial_{x_4},
\]
the third layer by $X_4=[X_1,X_3]= \partial_{x_3}+ x_1 \partial_{x_4}$ and the fourth layer by  $X_5=[X_1,X_4]=\partial_{x_4}$. Now the curve $\ga: \rr \to \rr^5$ parametrized by $\ga(t)=(0,t,0,0,0)$ is a curve of degree $2$, its tangent vector $\ga'(t)$ is given by $\partial_{x_2}$. Therefore we have $n_2=3$. We consider the Riemannian metric $g=\escpr{\cdot, \cdot}$ that makes $(X_1,\ldots,X_5)$ an orthonormal basis. Due to the computation provided in Example \ref{ex:hsc} we deduce that 
\[
A= \left(\begin{array}{ccc} -1 & 0 & 0\\
                                           0 & 0 &0 
             \end{array} \right)
\quad 
B= \left(\begin{array}{cc} 0 & 0\\
                                        0 &0 
             \end{array} \right).            
\]
Since $A$ is not linearly full we conclude that $\ga$ is a non-regular curve of degree $2$.
\end{example}

\begin{example}
Let $\mh$ be $3$-dimensional distribution on $\rr^5$ generated by 
\[
X_1=\partial_{x_1} , \quad X_2= \partial_{x_2} + x_ 1\partial_{x_3}+ \dfrac{x_1^2}{2} \partial_{x_4} \quad \text{and} \quad X_3=\partial_5.
\] 
The second layer is generated by $X_4=[X_1,X_2]$ and the third layer by $X_5=[X_1,X_4]$. Let $\ga: \rr \to \rr^5$ be the horizontal curve parametrized by $\ga(t)=(0,0,0,0,t)$ whose tangent vector $\ga'(t)$ is given by $X_3=\partial_{x_5}$. We consider the Riemannian metric $g=\escpr{\cdot, \cdot}$ that makes $(X_1,\ldots,X_5)$ an orthonormal basis. Since $X_3$ does not commute with any vector fields of the basis we deduce that $A=0$ and $B=0$. Then the admissibility equation \eqref{eq:compmatrix} is given $F'=0$ with initial condition $F(a)=0$. Clearly the holonomy map is not surjective since $F(b)=0$.  Therefore the curve $\ga$ is non-regular on each subinterval $[a,b]\subset \rr$.
\end{example}

\begin{example}[Kolmogorov]
Let us consider in $\rr^4$ the Kolmogorov operator
\[
 L=\partial_t + x \, \partial_y+ \dfrac{x^2}{2} \, \partial_z- \partial_{xx}^2,
\]
where $(x,y,z,t)$ is a point in $\rr^4$. Notice that $L$ is homogenous of degree two under the dilation
$\delta_{\lambda}((x,y,z,t))=(\lambda x,\lambda^3 y,\lambda^4 z,\lambda^2 t)$. Therefore the graded structure adapted to $L$ is given by $\rr^4$ endowed with the filtration
\begin{align*}
\mh^1&= \text{span}\{X_1=\partial_x\}\\
\mh^2&=\text{span}\{X_1,X_2=\partial_t + x \, \partial_y+ \tfrac{x^2}{2} \, \partial_z\}\\
\mh^3&=\text{span}\{X_1, X_2, X_3=[X_1,X_2]=\partial_y+ x\partial_z \}\\
\mh^4&=\text{span}\{X_1, X_2, X_3, X_4=[X_1,X_3]=\partial_z\}.
\end{align*}
Setting $\mh:=\mh^2$, we allow only curve of degree less than or equal to two. Due to the computations developed in Example \ref{ex:hsc}, we obtain that $\ga(s)=(0,0,0,s)$ is singular of degree two. 
\end{example}

\begin{example}
\label{ex:heis1}
The $3$-dimensional Heisenberg group $\mathbb{H}^1$ is a Lie group defined by a Lie algebra $\mathfrak{h}$ generated $\{X,Y,T\}$, where the only non-trivial relation is $T=[X,Y]$. Setting $\mh=\text{span}\{X,Y\}$, $(\mathbb{H}^1, \mh)$ is the simplest example of Carnot group, that clearly is a Carnot manifold. A possible presentation of the Heisenberg group is provided by $\rr^3$ where the vector field $\{X,Y,T\}$ are given by
\[
X=\dfrac{\partial}{\partial {x}}-\dfrac{y}{2}\dfrac{\partial}{\partial_t}, \quad
Y=\dfrac{\partial}{\partial {y}}+\dfrac{x}{2}\dfrac{\partial}{\partial_t} \quad 
T=\dfrac{\partial}{\partial_t}.
\]
As we point out in Section \ref{sc:submanifold} each $C^1$  surface $\Sigma$ immersed in the  $\mathbb{H}^1$ inherits a structure of graded manifold $(\Sigma, \tilde{\mh}^1,\tilde{\mh}^2)$, where $\tilde{\mh}^1=\mh \cap T \Sigma$ and $\tilde{\mh}^2= T \Sigma$. Moreover $\Sigma\smallsetminus \Sigma_0$, where $\Sigma_0$ denotes the characteristic set, is an equiregular graded manifold. The foliation properties by horizontal  integral curves of $\tilde{\mh}^1$ have been deeply studied by \cite{scottpauls,GarofaloNh,ChengMalchiodi}. We notice that each integral curve $\ga: I \to \Sigma$ of  $\tilde{\mh}^1$ is singular restricted to each $[a,b]\subset I$. Since the $\dim(\tilde{\mh}^1)=1$ the one dimensional matrix $A(t)=0$ and the admissibility system \eqref{eq:compmatrix} is given by 
\[
F'=BF.
\]
Fixing the initial condition $F(a)=0$ the unique solution of the homogeneous system is $F(t)\equiv0$, thus the holonomy map is not surjective.

\end{example}

\begin{remark}
Let $\partial_{x_1},\ldots,\partial_{x_n}$ be the Euclidean basis in $\rr^n$, $n\ge2$. Let $1\le k \le n$. Assume that we set $\mh^1=\text{span}\{\partial_{x_1},\ldots,\partial_{x_k}\} $ and $\mh^2=T\rr^n$. Let $\ga:\rr \to \rr^n$ be a curve in $\mh^1$ such that $\ga'(t)=\sum_{\ell=1}^k h_{\ell} \, \partial_{x_{\ell}}$. Let $\escpr{\cdot,\cdot}$ be the standard Euclidian metric in $\rr^n$. Setting $\mh=\mh^1$ we obtain that matrices defined  in \eqref{def:AB} is given by
\[
a_{r i}=\sum_{\ell=1}^k h_{\ell} \escpr{[\partial_{x_\ell}, \partial_{x_i}], \partial_{x_r}}=0
\]
and 
\[
b_{r j}=\sum_{\ell=1}^k h_{\ell} \escpr{[\partial_{x_\ell}, \partial_{x_j}], \partial_{x_r}}=0,
\]
for all $i=1,\ldots,k$ and $r,j=k+1,\ldots,n$. Since $A=0$ and $B=0$ we deduce that $F'(t)=0$ and $F(a)=0$. Therefore $F(b)=0$ and the holonomy map is not surjective. Hence in this setting each horizontal curve is singular\end{remark}

\section{A new integrability criterion for admissible vector fields}
\label{sc:weakint}
In this subsection we give a sufficient condition for a horizontal curve $\ga:I\to N$ to be regular in $[a,b]\subset I$. The condition is that the matrix $A(t)$ associated to the admissibility system of differential equations \eqref{eq:compmatrix}, defined in \eqref{def:AB}, has rank $(n-k)$ for any $t\in I$. By Proposition~\ref{prop:linearlyfull}, this condition implies that the curve $\ga$ is regular in $[a,b]$. The aim of this section is to give a direct proof of this fact, \emph{that generalizes to higher dimensions}.

We consider the following spaces:
\begin{enumerate}
\item $\mathfrak{X}^r(I,N)$, $r\ge0$: the set of $C^r$ vector fields along $\ga$.
\item $\mh^r(I,N)$, $r\ge0$: the set of $C^r$  horizontal vector fields along $\ga$.
\item $\mathcal{V}^r(I,N):=\{Y\in \mathfrak{X}^r(I,N): \escpr{Y,X}=0 \ \forall X \in \mh^r(I,N) \}=\mh^r(I,N)^{\perp}$.
\end{enumerate} 
We shall denote by $\Pi_v$ the orthogonal projection over the vertical subspace.

As in the previous subsections we consider an orthonormal adapted basis $(X_i)$ along $\gamma$ and the associated admissibility system in matrix form $F'=BF+AG$, where the matrices $A$, $B$, $F$ and $G$ are defined in \eqref{def:AB} and \eqref{def:FG}.

\begin{definition}
\label{def:stronglyregular}
We say that a horizontal curve $\ga: I \to N$ is \emph{strongly regular} at $t\in I$ if $\text{rank}\, A (t) =n-k$. We say that $\ga$ is \emph{strongly regular} in $J\subset I$ if it is strongly regular at every $t\in J$.
\end{definition}

\begin{remark} \mbox{}
\begin{enumerate}
\item Given $t_0\in I$ such that $\ga$ is strongly regular at $t_0$, there exists a small neighborhood $J$ of $t_0$ in $I$ where the rank of $A(t)$ is given by a fixed subset of columns of $A(t)$ for all $t\in J$. This neighborhood $J$ can be extended to a maximal one where this property holds.
\item Notice that a curve $\ga$ can be strongly regular at $t \in I$ only when $k \ge \frac{n}{2}$.
\end{enumerate}
\end{remark}

The third equation in \eqref{eq:tilde} implies that this definition is independent of the chosen adapted orthonormal basis along $\ga$. With this definition, we are able to prove

\begin{lemma}
\label{lm:kerA}
Let $k\ge \frac{n}{2}$. Let $A(t)$ be the $C^r$ $(n-k) \times k$ matrix defined in \eqref{def:AB} with respect to $(X_1,\ldots,X_n)$.  Assume that $\text{\emph{rank}}\, A (t) =n-k$ for $t\in [a,b]$. Then there exists a horizontal orthonormal global basis $\tilde{X}_1,\ldots,\tilde{X}_k$ on $[a,b]$ such that the matrix $\tilde{A}(t)$ with respect to the orthonormal global basis $(\tilde{X}_1,\ldots,\tilde{X}_k, X_{k+1},\ldots,X_n)$ is given by
\[
\tilde{A}(t)=\big( \tilde{A}_1(t)\ \ 0 \big),
\]
 where $\tilde{A}_1(t)$ is an invertible square $(n-k)$ matrix.
\end{lemma}
\begin{proof}
The proof is by induction on the dimension of the kernel of $A$ that is equal to $2k-n$. 

When $\dim(\ker A(t))=1$ locally there exist two unitary vector fields $X(t)$ and $-X(t)$ in the kernel of $A(t)$. We  define a global vector field $\tilde{X}_{k}(t)$ locally choosing one of the two unit vectors $X(t)$ or $-X(t)$ in the kernel and adjusting them in the overlapping intervals. Then we  extend the unitary vector field $\tilde{X}_k$ to an orthonormal horizontal basis $(\tilde{X}_1, \ldots, \tilde{X}_k)$. Therefore with respect to $(\tilde{X}_1, \ldots, \tilde{X}_k, X_{k+1},\ldots,X_n)$ the last column of the matrix $\tilde{A}$ is equal to zero.

If $\dim(\ker A(t))>1$, fix $\bar{t} \in ]a,b[$, then by a continuation argument for the determinant there exists an open neighborhood $U_{\bar{t}}=]\bar{t}-\delta,\bar{t}+\delta[$ and a non vanishing $C^r$ vector field $V(t)$ on $U$ such that $A(t) V(t)=0$. Then $\{U_{t}\}_{t \in [a,b]}$ is an open cover of the compact set $[a,b]$ then there exists a finite sub-cover $U_1,\ldots,U_Q$ such that $U_{\alpha}\cap U_{\beta} \ne \emptyset$ for $\alpha, \beta \in \{1,\ldots,Q\}$, $\alpha< \beta$ if and only if $\beta=\alpha + 1$. Let $\{\psi_{\alpha} \, : \,  \alpha \in  \{1, \ldots,Q\} \}$ be a partition of unity subordinate to the cover $\{U_{\alpha} \ : \ \alpha \in  \{1, \ldots,Q\}\}$  for further details see \cite[Definition 1.8]{warner}. For each $\alpha$ there exists a non vanishing $C^r$ vector fields $V^{\alpha}(t)$ on $U_{\alpha}$ in $\ker A(t)$. When $U_{\alpha} \cap U_{\beta} \ne \emptyset$ we consider $V^{\alpha}(t)$ on $U_{\alpha}$ and $V^{\beta}(t)$ on $U_{\beta}$ in $\ker A(t)$ such that they are linear independent on $U_{\alpha} \cap U_{\beta}$, since the $\dim(\ker A(t))>1$. Then, we set   
\[
X(t)=\sum_{\alpha=1}^Q \psi_{\alpha}(t) V^{\alpha}(t).
\]
Therefore $X(t)$ is a global non vanishing vector field that belongs to $\ker (A(t))$ for all $t \in [a,b]$. Thus we can extend  the global unitary vector 
\[
\tilde{X}_k(t)= \dfrac{X(t)}{|X(t)|}
\]
to an orthonormal basis of the horizontal distribution. The matrix associated to this basis has a vanishing last column. We remove this column and start again until we have dimension $1$. Hence in this new global horizontal basis $(\tilde{X}_1,\ldots, \tilde{X}_k)$ the last $2k-n$ columns of the matrix $\tilde{A}(t)$ vanish and the rank is concentrated in the square matrix $A_1(t)$ given by the first $n-k$ columns.
\end{proof}

\begin{theorem}
\label{thm:intcriterion2}
Let $\ga:I\to N$ be a horizontal curve in a Carnot manifold $(N,\mathcal{H})$ endowed with a Riemannian metric. Assume that $\ga$ is strongly regular in $[a,b]\subset I$. Then every admissible vector field with compact support in $(a,b)$ is integrable.
\end{theorem}

\begin{proof}
Let $J=[a,b]$. The admissibility system is given  by
\begin{equation}
\label{eq:admissibility-1}
F'(t)+B(t)F(t)+A(t)G(t)=0,
\end{equation}
with respect to $(X_i)$.
By hypothesis, the rank of $A(t)\in C^r$ is maximal for all $t\in J$. By Lemma~\ref{lm:kerA} there exists a global basis $(\tilde{X}_i)$ such that
 \[
\tilde{A}(t)=\big(\tilde{A}_1(t)\ \ 0 \big),
\]
where $\tilde{A}_1(t)$ is an invertible square $(n-k)$ matrix.
Then setting $\tilde{F}, \tilde{G}$ the new coordinates with respect to $({\tilde{X}}_i)$ and $\tilde{B}$ as in \eqref{eq:tilde}, by Remark~\ref{rk:systemtilde} we have 
\begin{equation}
\label{eq:admtilde}
\tilde{F}'+\tilde{B}\tilde{F}+\tilde{A}\tilde{G}=0
\end{equation}
Calling
\[
\tilde{G}_1=\begin{pmatrix}
\tilde{g}_1 \\ \vdots \\ \tilde{g}_{n-k}
\end{pmatrix},
\quad
\tilde{G}_2=\begin{pmatrix}
\tilde{g}_{n-k+1}\\ \vdots \\ \tilde{g}_{k}
\end{pmatrix},
\]
the admissibility system \eqref{eq:admtilde} can be written as
\[
\tilde{F}'+\tilde{B}\tilde{F}+\tilde{A}_1\tilde{G}_1=0,
\]
and so
\begin{equation}
\label{eq:admissibility-G1}
\tilde{G}_1=-\tilde{A}_1^{-1}\big(\tilde{F}'+\tilde{B} \tilde{F}\big).
\end{equation}

Now let $\mathcal{H}_1^r(J)$ be the set of horizontal vector fields of class $C^r$  in $J$ that are linear combination of the first $(n-k)$ vectors ${\tilde{X}}_1,\ldots,{\tilde{X}}_{n-k}$, and $\mathcal{H}_2^r(J)$ the set of horizontal vector fields of class $C^r$ in $J$ that are linear combination of the vector fields ${\tilde{X}}_{n-k+1},\ldots,{\tilde{X}}_k$.

We consider the map
\begin{equation}
\label{eq:defG}
\mathcal{G}: \mathcal{V}^{r}(J)\times\mathcal{H}_1^{r-1}(J) \to \mathcal{V}^{r}(J)\times\mathcal{V}^{r-1}(J)
\end{equation}
defined by
\[
\mathcal{G}(Y_1,Y_2)=(Y_1,\mathcal{F}(Y_1+Y_2)).
\]
We recall that, given a vector field $Y$ along a portion of $\ga$, we define the curve $\Gamma(Y)(t)$ by $\exp_{\gamma(t)}(Y(t))$ and we define $\mathcal{F}(Y)$ as the vertical projection of the tangent vector $\Gamma(Y)'$. We consider on each of the spaces appearing in \eqref{eq:defG} the corresponding $||\cdot||_r$ or $||\cdot||_{r-1}$ norm, and in the product one of the classical product norms.

Then
\[
D\mathcal{G}(0,0)(Y_1,Y_2)=(Y_1,D\mathcal{F}(0)(Y_1+Y_2)),
\]
where $D\mathcal{F}(0)Y$ is given by
\[
D\mathcal{F}(0)Y= \sum_{r=k+1}^{n} \left(\tilde{f}_r'(t)+\sum_{i=1}^{n-k} \tilde{a}_{r i} (t) \tilde{g}_i (t)+\sum_{j=k+1}^n \tilde{b}_{r j}(t) f_j(t) \right) X_{r}.
\] 
Observe that $D\mathcal{F}(0)Y=0$ if and only if $Y$ is an admissible vector field, namely $Y$ solves \eqref{eq:admissibility-1}.

\medskip

Our objective now is to prove that the map $D\mathcal{G}(0, 0)$ is an isomorphism of Banach spaces. Indeed suppose that $D\mathcal{G}(0, 0)(Y_1,Y_2)=(0,0)$. This implies that $Y_1$ is  equal to zero. In the coordinates previously described, $Y_1$ to $\tilde{F}$, and $Y_2$ to $\tilde{G}_1$. By the admissibility equation \eqref{eq:admissibility-G1} we have that also $Y_2$ is equal to zero. This proves that $D\mathcal{G}(0, 0)$ is injective. Let us prove now that $D\mathcal{G}(0,0)$ is surjective. Take $(Z_1, Z_2)$, where  $ Z_1 \in  \mathcal{V}^r(J) $, and $ Z_2 \in \mathcal{V}^{r-1}(J)$ we seek $Y_1,Y_2$ such that $D\mathcal{G}(0, 0)(Y_1,Y_2)=(Z_1, Z_2)$. Then $Y_1=Z_1$ and $Y_2$ is obtained by solving the system 
\[
\tilde{G}_1=-\tilde{A}_1^{-1}\big(\tilde{F}'+\tilde{B} \tilde{F} + \tilde{Z} \big),
\]
since $Y_1$ and $Z_2=\sum_{r=k+1} \tilde{z}_r \tilde{X}_r$ is already given. This proves that $D\mathcal{G}(0,0)$ is surjective.

Keeping the above notation for $Y_i, Z_i$, $i=1,2$, we notice that $D\mathcal{G}(0, 0)$ is a continuous map since the identity map is continuous and there exists a constant $K$ such that 
\begin{align*}
\| Z_2 \|_{r-1} &\le K \left( \left\| \dfrac{d}{dt}(Y_1 ) \right\|_{r-1}+ \|Y_1\|_{r-1}+ \|Y_2\|_{r-1} \right)\\
                      &\le K( \| Y_1\|_{r} + \|Y_2\|_{r-1} ).
\end{align*}
Moreover, $D\mathcal{G}(0, 0)$ is an open map since we have 
\begin{align*}
\| Y_2 \|_{r-1} &\le K \left(  \left\| \dfrac{d}{dt}(Z_1 ) \right\|_{r-1}+ \|Z_1\|_{r-1}+ \|Z_2\|_{r-1} \right)\\
                      &\le K ( \| Z_1\|_{r} + \|Z_2\|_{r-1} ).
\end{align*}
This concludes the proof that $D\mathcal{G}(0,0)$ is an isomorphism of Banach spaces.

\medskip

Let us finally consider an admissible vector field $V$ compactly supported on $(a,b)$. We consider the map
\[
\tilde{\mathcal{G}}:(-\eps,\eps)\times \mathcal{V}^{r}(J)\times\mathcal{H}_1^{r-1}(J) \to \mathcal{V}^{r}(J)\times\mathcal{V}^{r-1}(J),
\]
defined by
\[
\tilde{\mathcal{G}}(s,Y_1,Y_2)=(Y_1,F(sV+Y_1+Y_2)).
\]
The map $\tilde{\mathcal{G}}$ is continuous with respect to the product norms (on each factor we put the natural norm, the Euclidean one on the intervals and $||\cdot||_r$ and $||\cdot||_{r-1}$ in the spaces of vectors along $\ga$). Moreover
\[
\tilde{\mathcal{G}}(0,0,0)=(0,0),
\]
since $\ga$ is horizontal. Now we have that
\[
D_2\tilde{\mathcal{G}}(0,0,0)(Y_1,Y_2)=D\mathcal{G}(0,0)(Y_1,Y_2)
\]
is a linear isomorphism. We can apply the Implicit Function Theorem to obtain $\eps>0$ and unique maps
\[
Y_1:(-\eps,\eps) \to \mathcal{V}^{r}(J), \quad Y_2:(-\eps,\eps) \to \mathcal{H}_{1}^{r-1}(J)
\]
such that $\tilde{\mathcal{G}}(s,Y_1(s),Y_2(s))=(0,0,0)$. This implies that $Y_1(s)\equiv 0$, $Y_2(0)=0$ and that
\[
\mathcal{F}(sV+Y_2(s))=0.
\]
Hence the curve $\Ga_s:J\to N$ defined by $\Gamma_s(t)=\exp_{\ga(t)}(sV(t)+Y_2(s)(t))$ is horizontal for $s\in (-\eps,\eps)$. Differentiating the above formula at $s=0$ we obtain 
\[
D\mathcal{F}(0)\left( V+\dfrac{\partial Y_2}{\partial s} (0)\right)=0.
\]
As $V$ is admissible we deduce that
\[
D\mathcal{F}(0)\bigg( \dfrac{\partial Y_2}{\partial s} (0)\bigg) =0.
\]
Since\[
\dfrac{\partial Y_2}{\partial s} (0)=\sum_{h=1}^{n-k} \tilde{f}_{i} \tilde{X}_{i}, \quad \tilde{f}_{i} \in C^{r-1}(J),
\]
we get from by equation \eqref{eq:admissibility-G1}  that $	\tilde{f}_{i} \equiv 0$ for each $h=1,\ldots,n-k$. Therefore it follows that $\tfrac{\partial Y_3}{\partial s} (0)=0$. This way we obtain that the variational vector field of the variation $\Gamma_s$ is
\[
\dfrac{\partial \Gamma_s}{\partial s} \bigg|_{s=0}= V+ \dfrac{\partial Y_2}{\partial s} (0)=V.
\]
The uniqueness property of the Implicit Function Theorem provides $Y_i(s)=0$ for $s\in (-\eps,\eps)$ when $V=0$.
\end{proof}

\section{The first variation formula}
\label{sc:fvf}

Let $\ga:I \to N$ be a curve of degree $d$ in a Carnot manifold endowed with a Riemannian metric. We fix an orthonormal adapted basis $(X_1, \ldots, X_n)$ along the curve. Recall that the length of degree $d$ was computed in \eqref{eq:integral_formula_Ad} as $L_d(\ga,J)=\int_J \theta_d(t)dt$, where the \emph{length density} $\theta_d$ of degree $d$ is given by 
\[
\theta_d(t)=\bigg( \sum_{j=n_{d-1}+1}^{n_d} \escpr{\ga'(t), (X_j)_{\ga(t)}}^2 \bigg)^{\frac{1}{2}}.
\]
In case $\theta_d(t)\neq 0$ for any $t\in I$, we can reparameterize $\ga$ so that the new length density is identically $1$.

Let $\Ga(t,s)$ be an admissible variation of $\ga$ whose variational vector field is given by $V(t)=\tfrac{\ptl\Ga}{\ptl s}(t,0)$. Calling $\theta=\theta_d$, the derivative of the length functional $L_d$ is given by

\begin{align*}
\dfrac{d}{ds}\Big|_{s=0} L_d(\Gamma_s,I)&=\dfrac{d}{ds}\Big|_{s=0} \int_I \bigg( \sum_{j=n_{d-1}+1}^{n_d} \left\langle \dfrac{\Gamma(t,s)}{\partial t}, (X_j)_{\Gamma(t,s)} \right \rangle^2 \bigg)^{\frac{1}{2}} dt\\
&=\sum_{j=n_{d-1}+1}^{n_d} \int_I  \,  \dfrac{\escpr{\ga'(t), X_j}}{\theta(t)}  \,   \dfrac{d}{ds}\Big|_{s=0}  \left\langle \dfrac{\Gamma(t,s)}{\partial t}, (X_j)_{\Gamma(t,s)}\right \rangle \, dt \\
&=\sum_{j=n_{d-1}+1}^{n_d}  \int_I \,   \dfrac{\escpr{\ga'(t), X_j}}{\theta(t)}  \,  (  \escpr{\nabla_{\ga'} V(t), X_j}+\escpr{\ga'(t), \nabla_{V(t)} X_j} ) \, dt.
\end{align*}
Integrating by parts we obtain that
\begin{equation*}
 \int_I \,   \dfrac{\escpr{\ga'(t), X_j}}{\theta(t)}   \escpr{\nabla_{\ga'} V(t), X_j}\,dt=\int_I \,   -   \escpr{ V(t),\nabla_{\ga'} \left(\dfrac{\escpr{\ga'(t), X_j}}{\theta(t)} X_j \right)} \, dt.
\end{equation*}
Since $V=\sum_{i=1}^n \escpr{V,X_i}X_i$ we get
\[
\escpr{\ga'(t),\nabla_V X_j}=\escpr{V,\sum_{i=1}^n\escpr{\ga',\nabla_{X_i}X_j}X_i},
\]
and so we can write
\begin{equation}
\frac{d}{ds}\bigg|_{s=0} L_d(\Ga_s,I)=\int_I\escpr{V,\mathbf{H}}\,dt,
\end{equation}
where
\begin{equation}
\mathbf{H}= \sum_{j=n_{d-1}+1}^{n_d}\bigg(-\nabla_{\ga'}\bigg(\frac{\escpr{\ga'(t),X_j}}{\theta(t)}X_j\bigg)+\sum_{i=1}^n \frac{\escpr{\ga'(t),X_j}}{\theta(t)}\escpr{\ga',\nabla_{X_i}X_j}X_i\bigg).
\end{equation}
Expressing $\ga'=\sum_{\ell=1}^{n_d} \escpr{\ga',X_\ell}X_\ell$, we can rewrite $\mathbf{H }$ as
\begin{equation}
\label{eq:fvf}
\mathbf{H}=\!\!\!\!\sum_{j=n_{d-1}+1}^{n_d}\bigg(- \dfrac{d}{dt}\left(\dfrac{\escpr{\ga'(t), X_j}}{\theta(t)}\right) X_j+\sum_{i=1}^n\sum_{\ell=1}^{n_d}\frac{\escpr{\ga'(t),X_j}\escpr{\ga'(t),X_\ell}}{\theta(t)}\,c_{\ell i}^j X_i\bigg),
\end{equation}
where
\begin{equation}
\label{eq:cellij}
c_{\ell i}^j=-\escpr{\nabla_{X_\ell}X_j,X_i}+\escpr{X_\ell,\nabla_{X_i}X_j}.
\end{equation}
With this preparation, we can compute the first variation formula for the length $L_d$ of degree $d$ for regular curves.

\begin{remark}
Let $m$ be a positive integer. We remind that when $x$ is a vector field in $\rr^m$ with coordinates $x_i$ for $i=1,\ldots,m$, its transpose is the row vector  
$x^T=(x_1,\ldots,x_m)$.
\end{remark}

\begin{theorem}
\label{thm:geod}
Let $\ga: I \to N$ be a curve of degree $d=\deg(\ga)$  such that $\theta_d(t)=1$ for each $t \in I\smallsetminus I_0 $. Assume that the curve $\ga$ is regular  restricted to $[a,b] \subset I \smallsetminus I_0$. Then $\ga$ is a critical point of the length of degree $d$ for any admissible variation if and only if there is a constant vector $k^T \in \rr^{n-n_d}$ such that $\ga$ satisfies along $\ga$ the following differential equation
\begin{equation}
\label{eq:geod}
-\dot{\alpha}^T(t)   + \beta_h^T (t) = \left(k - \left( \int_a^t \beta_v^T (\tau) D^{-1}(\tau) \, d\tau \right)  \right) D(t) A(t),
\end{equation}
where $h_i=\escpr{\ga', X_i}$ for $i=1,\ldots,n_d$, $A(t)$ defined in \eqref{def:AB}, $D(t)$ solving \eqref{eq:homD} and we set
\begin{equation}
\label{def:alphabeta}
\alpha=\begin{pmatrix}
			0\\
			\vdots\\
			0\\
			h_{n_{d-1}+1}\\
			\vdots\\
			h_{n_d}
		\end{pmatrix}, \qquad 
   \beta=\begin{pmatrix} \beta_h \\ \beta_v \end{pmatrix} =\begin{pmatrix}
			\beta_1\\
			\vdots\\
			\beta_{n_d}\\
			\beta_{n_{d+1}}	\\
			\vdots\\
			\beta_{n}
	 	\end{pmatrix},
\end{equation}
with 
\[
\beta_i=\sum_{\ell=1}^{n_d} \sum_{j=n_{d-1}+1}^{n_d} h_j \, h_{\ell}  \, c_{\ell i}^j.
\]
\end{theorem}

\begin{proof}
Fix an adapted basis $(X_i)$ along $\ga$, and consider the admissible vector field
\[
V=\sum_{i=1}^{n_d} g_i X_i + \sum_{j=n_d + 1}^n f_i X_i
\]
solving the admissibility equation \eqref{eq:compmatrix}. Then we have
\begin{equation}
\label{eq:admisswithD}
(DF)'=-D A G.
\end{equation}

Since $\ga$ is regular by hypothesis, Theorem~\ref{thm:intcriterion} implies that $V$ is integrable, and so there exists an admissible variation $\Ga(t,s)$ such that $V(t)=\tfrac{\ptl\Ga}{\ptl s}(t,0)$.
With the notation introduced in \eqref{def:alphabeta} the first variational formula with respect to $(X_i)$ is given by 
\begin{equation}
\label{eq:fvf2}
\dfrac{d}{ds}\Big|_{s=0} L_d(\Gamma_s,I)=\int_I -\dot{\alpha}^T (t) G(t)+ \beta_h^T (t) G(t)+  \beta_v^T (t) F(t) \, dt.
\end{equation}
Since \eqref{eq:admisswithD} holds and $F$ is compactly supported in $[a,b]$ we have
\begin{align*}
\int_I \beta_v^T (t) F(t) \, dt&= \int_I \beta_v^T (t) D^{-1}(t) D(t) F(t) \, dt\\
&= -\int_I  \left( \int_a^t \beta_v^T (\tau) D^{-1}(\tau) \, d\tau \right)  (D F)'(t) \, dt\\
&= \int_I  \left( \int_a^t \beta_v^T (\tau) D^{-1}(\tau) \, d\tau \right)  D(t) A(t) G(t)\, dt.
\end{align*}
Therefore \eqref{eq:fvf2} is equivalent to 
\begin{equation}
\label{eq:fvf3}
\int_I \left( -\dot{\alpha}^T(t) + \left( \int_a^t \beta_v^T (\tau) D^{-1}(\tau) \, d\tau \right)  D(t) A(t)+ \beta_h^T (t)  \right) G(t) dt,
\end{equation}
for each $G(t)$ that verifies
\[
F(b)=-D(b)^{-1} \int_a^b D(t) A(t) G(t) dt =0.
\]
Hence a critical point of the functional $L_d$ is given by 
\[
\int_I \left( -\dot{\alpha}^T(t)+  \left( \int_a^t \beta_v^T (\tau) D^{-1}(\tau) \, d\tau \right)  D(t) A(t)+ \beta_h^T (t)  \right) G(t) dt=0,
\]
for each $G$ satisfying
\[
H^{a,b}_{\ga}(G)=D(b)^{-1} \int_a^b D(t) A(t) G(t) dt =0.
\]
Since the holonomy map is surjective by the du Bois-Reymond Lemma  \cite[Lemma C.1]{MR1189496}  there exists a constant vector $\tilde{k}^T \in \rr^{n-n_d}$ such that the Euler-Lagrange equation is given by 
\begin{equation}
\label{eq:geod2}
-\dot{\alpha}^T(t)+\left( \int_a^t \beta_v^T (\tau) D^{-1}(\tau) \, d\tau \right)  D(t) A(t)+ \beta_h^T (t) = k \ D(b)^{-1} D(t) A(t).
\end{equation}
Since $D(b)^{-1}$ is a constant matrix we have that  $k=\tilde{k} D(b)^{-1}$ is a constant row vector. Hence $\ga$ satisfies equation \eqref{eq:geod}.

Conversely, assume that equation \eqref{eq:geod} holds so than also \eqref{eq:geod2} holds. Putting equation \eqref{eq:geod2} into \eqref{eq:fvf3} we obtain 
\begin{equation*}
\dfrac{d}{ds}\Big|_{s=0} L_d(\Gamma_s,I)=\int_I k \ D(t) A(t)G(t) dt= k  \int_I D(t) A(t)G(t) dt.
\end{equation*}
Since \eqref{eq:admisswithD} holds and $F$ is compactly supported in $[a,b]$ we conclude
\[
\dfrac{d}{ds}\Big|_{s=0} L_d(\Gamma_s,I)=k \ \left( D(t) F(t) \right)\Big|_{t=a}^{t=b}=0.
\]
Hence we proved that $\ga$ is a critical point of the length functional of degree $d$ for any admissible variation.
\end{proof}

\begin{example}[The Heisenberg group $\hh^n$]
\label{ex:Heisenberg}
A well-known example of contact sub-Riemannian manifold (see Example \ref{ex:regularcontact})  is the Heisenberg group $\mathbb{H}^n$, defined as $\rr^{2n+1}$ endowed with the contact form 
\[
\omega_0=dt+\sum_{i=1}^n (x_i dy_i-y_i dx_i).
\]
Moreover $\mathbb{H}^n$ is a Lie group $(\rr^{2n+1},*)$   where the product is defined, for any pair of points $(z,t)=(z_1,\ldots,z_{n},t)$, $(z',t')=(z_1',\ldots,z_{n}',t')$ in $\rr^{2n+1}=\cc^{2n}\times \rr$, by 
\[
(z,t)*(z',t')=\Bigg(z+z',t+t'+\sum_{i=1}^n \text{Im}(z_i\bar{z_i}')\Bigg).
\]
A basis of left invariant vector fields is given by $\{X_1,\cdots,X_n,Y_1, \cdots,Y_n,T\}$, where
\[
X_i=\dfrac{\partial}{\partial {x_i}}-\dfrac{y_i}{2}\dfrac{\partial}{\partial_t}, \quad
Y_i=\dfrac{\partial}{\partial {y_i}}+\dfrac{x_i}{2}\dfrac{\partial}{\partial_t} \quad i=1,\ldots,n, \quad 
T=\dfrac{\partial}{\partial_t}.
\]
The only non-trivial relation is $[X_i,Y_i]=T$. Here the horizontal metric $h$ is the one that makes 
$\{X_i,Y_i : i=1,\cdots,n\}$ an orthonormal basis of $\mathcal{H}=\text{ker}(\omega_0)$. On the tangent bundle we consider the metric $g=\escpr{\cdot, \cdot}$ so that \eqref{eq:contactmetric} holds. Clearly, we have $\escpr{X_i,T}=\escpr{Y_i,T}=0$ for all $i=1,\ldots,n$. Let $\nabla$ be the Levi-Civita connection associated to $g$. From Koszul formula and the Lie bracket relations we get
\begin{equation}
\label{eq:kozul}
\nabla_{X_i} X_j = \nabla_{Y_i} Y_j = \nabla_T T= 0, \quad \nabla_{X_i} Y_j = -\delta_{ij} T, \quad 
\nabla_{Y_i} X_j = \delta_{ij} T\\
\end{equation}
For any vector field $X$ on $\hh^n$ we have $J(X)=\nabla_X T$. Following the previous notation we set $X_{n+i}:=Y_{i}$ for all $i=1,\ldots,n$ and $X_{2n+1}:=T$, then the only non-trivial structure constants are 
\begin{equation}
\label{eq:strconh}
c_{i \, n+i}^{2n+1}= -c_{n+i \, i}^{2n+1}=\escpr{[X_i, X_{i+n}], X_{2n+1}}=1,
\end{equation}
for all $i=1,\ldots,n$. 

Let $\ga:I \to \hh^n$ be an horizontal curve parameterized by arc length, i.e. $\theta(t)=1$.  By equation \eqref{eq:kozul} and the linearity of $\nabla$ on the first term we have $\escpr{\ga', \nabla_{\ga'} X_j}=0$ thus it holds
\[
\dot{\alpha}_{j}= \escpr{\nabla_{\ga'} \ga',X_j}  +\escpr{\ga', \nabla_{\ga'} X_j}= \escpr{\nabla_{\ga'} \ga',X_j},
\]
where $\alpha$ is defined in Theorem \ref{thm:geod}.
Since $\escpr{\nabla_{\ga'} \ga', T}=-\escpr{\ga', J(\ga')}=0$, then $\nabla_{\ga'} \ga'$ is horizontal. By equation \eqref{eq:strconh} we have $c_{\ell i}^j=0$ for all $j=1,\ldots,2n$, thus we deduce that 
\[
\beta_i=\sum_{\ell=1}^{2n} \sum_{j=1}^{2n} \escpr{\ga',X_j} \escpr{\ga',X_{\ell}} \, c_{\ell i}^j=0
\]
for all $i=1,\ldots,2n+1$. In this setting we have 
\[
B=\escpr{[\ga',T],T}=0, \quad A=(a_1,\ldots,a_{2n}),
\]
where $a_i=2\escpr{J(\ga'),X_i}$. Since the solution of following Cauchy problem
\[
\begin{cases}
D'(t)=D(t)B(t)=0\\
D(a)=1
\end{cases}
\]
is given by $D(t)=1$ for all $t \in [a,b]$, the right side term of \eqref{eq:geod} is given by $2 k J(\ga')$. Then  we conclude  that $\ga$ is a critical point of the horizontal length functional for any admissible variation if and only if there exists a constant $k \in \rr$ such that 
\begin{equation}
 -\nabla_{\ga'} \ga'=2 k J(\ga').
\end{equation}
Explicit solutions to this geodesic equation can be find in \cite[p. 10 ]{RitoreRosales1}, \cite[p. 160]{Monti2} and in \cite[p. 28]{Bellaiche}.

Let now $\ga: I \to \hh^n$ be a curve such that $\deg(\ga)=2$. We parametrize the curve $\ga$ so that  the  length density $\theta_2(t)=\escpr{\ga',T}=1$ for all $t \in I \smallsetminus I_0$. Since $\deg(\ga)=2$ is the maximal degree for a curve in Heisenberg the vertical set $\mathcal{V}_{\ga(t)}=\{0\}$ for all $t\in I \smallsetminus I_0$. Then $\ga$ is regular restricted to each interval $[a,b] \subset  I \smallsetminus I_0$. Therefore we have that $k=D=A=0 $ and 
\[
\alpha=\left(\begin{array}{c}
			0\\
			\vdots\\
			0\\
			1
		\end{array}
 \right),
 \quad 
\beta_i=\sum_{\ell=1}^{2n+1}  h_{2n+1} \, h_{\ell}  \, c_{\ell i}^{2n+1}.
 \]
Thus we have $\beta_i=-\escpr{\ga',Y_i}$ for $i=1, \ldots,n$, $\beta_i=\escpr{\ga',X_i}$ for $i=n+1, \ldots,2n$ and $\beta_{2n+1}=0$. We deduce $\beta=J(\ga')$. Hence the geodesic equation \eqref{eq:geod} is given by 
$
J(\ga')=0,
$
then $\ga'=T$. We conclude that the geodesics of degree $2$ are straight lines in direction $\partial_t$.
\end{example}

\subsection{Some properties of the length functional of degree two for surfaces immersed in the Heisenberg group}
Let  $\Sigma$ be a surface immersed in the Heisenberg group $\mathbb{H}^1$, where a basis of left-invariant vector fields is given by 
\[
 X=\partial_x+\dfrac{y}{2}\partial_t, \quad Y=\partial_y-\dfrac{x}{2}\partial_t, \quad T=\partial_t.
\]
We consider the ambient metric $g=\escpr{\cdot, \cdot}$ that makes $(X,Y,T)$ an orthonormal basis, see Example~\ref{ex:heis1}, and $\mh=\text{span}\{X,Y\}$. $\Sigma$  inherits the Riemannian metric $\bar{g}$ induced by $g$. Hence $(\Sigma,\tilde{\mh}^1, \tilde{\mh}^2)$ is a graded manifold endowed with the Riemannian metric $\bar{g}$, where $\tilde{\mh}_p^1=T_p {\Sigma}\cap \mh_p$, $\tilde{\mh}_p^2=T_p \Sigma$ if $p$ belongs to $\Sigma\smallsetminus \Sigma_0$ and $\tilde{\mh}_p^1=\tilde{\mh}_p^2=\mh_p=T_p \Sigma$ for $p\in \Sigma_0$.
Let $N$ be a unit vector normal to $\Sigma$ w.r.t. $g$ and 
$
 N_h=N-\langle N,T \rangle T
$
its orthogonal projection onto $\mh$.  In the regular part $\Sigma \smallsetminus \Sigma_0$, the horizontal Gauss map $\nu_h$ and the characteristic vector field $Z$ are defined by
\begin{equation}
\label{vectors nu Z}
  \nu_h= \dfrac{N_h}{|N_h|}, \quad Z=J(\nu_h),
\end{equation}
where $J(X)=Y$, $J(Y)=-X$ and $J(T)=0$. Clearly $Z$ is horizontal and orthogonal to $\nu_h$ then it is tangent to $\Sigma$.
If we define 
\begin{equation}
\label{vector S}
S=\langle N,T \rangle \nu_h-|N_h|T,
\end{equation}
then $(Z_p,S_p)=(e_1,e_2)$ is an orthogonal basis of $T_p \Sigma$ and it is adapted to the filtration $\mathcal{H}_p^1\cap T_p \Sigma \subset \mathcal{H}_p^2\cap T_p \Sigma$ for each $p$ in $\Sigma\smallsetminus \Sigma_0$.

In the regular part $\Sigma \smallsetminus \Sigma_0$ the length functional $L_2$ is well-defined. Since all variation $\Gamma_s$ are admissible and the first variation formula is given by 
\begin{equation}
\frac{d}{ds}\bigg|_{s=0}\L_2(\Ga_s,I)=\int_I\escpr{V,\mathbf{H}}\,dt,
\end{equation}
where $V$ is a vector field in $T\Sigma$ and $\mathbf{H}$ is given by equation \eqref{eq:fvf}. Then the Euler-Lagrange equation for $L_2$ is given by $\mathbf{H}=0$. Following equation \eqref{eq:fvf}, $\mathbf{H}=0$ is equivalent to 
\begin{equation}
\label{eq:geod2}
- \dfrac{d}{dt}\left(\dfrac{\escpr{\ga'(t), e_2}}{|\escpr{\ga'(t), e_2}|}\right) e_2+ \sum_{i=1}^2 \sum_{\ell=1}^2 \left( \frac{\escpr{\ga'(t),e_2} \escpr{\ga'(t),e_\ell}}{\escpr{\ga'(t),e_2} }\,c_{\ell i}^j\right) e_i=0.
\end{equation}
Then a straightforward computation shows that \eqref{eq:geod2} is equivalent to
\[
\left(\escpr{\ga'(t),e_1} c_{2 1}^2 \right) e_1+ \left(\escpr{\ga'(t),e_1} c_{1 2}^2  \right) e_2 =0
\]
This means that the geodesic equation for $L_2$ is given by 
\begin{equation}
\label{eq:geod2a}
\escpr{\ga'(t), Z} \escpr{ [S,Z] , S }=0.
\end{equation}
Whenever $\escpr{ [S,Z] , S }\ne0$ the unique geodesic for $L_2$ starting from $p$ is the integral curve of the vector field $S$ passing through $p$, namely the unique solution of the following Cauchy problem 
\[
\begin{cases}
 \ga'(t)=S_{\ga(t)}\\
 \ga(0)=p.
 \end{cases}
\]
The projection of the integral curve of $S$ onto the $xy$-plane are called  \emph{seed curves} in the literature, see for instance \cite[page 159]{MR2312336}.

\begin{example}
A vertical plane $P_v$  in $\mathbb{H}^1$ is given by 
\[
P_v=\{(x,y,t) \in \mathbb{H}^1 \ : \ ax+by=c , \ a^2 + b^2 =1, \ c \in \rr\}.
\]
It is easy to see that the $Z=bX-aY$ and $S=T$. Thus on a vertical plane  we always have $\escpr{ [T,Z] , T }=0$, since the Lie algebra of the Heisenberg group is nilpotent. More generally each surface obtained by the product of a planar curve in the $xy$-plane with $\rr$ in the $t$ direction (see \cite[Example 3.4]{MR3794892}) verifies $\escpr{ [S,Z] , S }=0$.  Therefore all curves in $P_v$ satisfy the geodesic equation \eqref{eq:geod2a}.
Now for seek of simplicity in the computation we consider the vertical plane $\{y=0\}$. This is not restrictive since a generic vertical plane $P_v$ can be obtained by a rotation and a left-translation of $\{y=0\}$. The length functional of degree $2$ is given by 
\[
L_2(\ga)=\int_a^b |\escpr{\ga'(s),T}| ds,
\]
where $\ga(s)=(x(s),t(s))$ is a piecewise $C^1$ curve in $P_v=\{y=0\}$. Let $p=(x_0,t_0)$ and $q=(x_1,t_1)$ be two points in $P_v$. We consider the piecewise curve  $\alpha(s): [0,2] \to P_v$   defined by 
\[
\alpha(s)=
\begin{cases}
\alpha_1(s)=\left(x_0, s\ t_1+ (1-s)t_0 \right)& \text{if} \quad s \in [0,1]\\
\alpha_0(s)= \left( (s-1) x_1 + (2-s) x_0, t_1\right)& \text{if} \quad s \in [1,2].
\end{cases}
\]
We claim that $\alpha(t)$ is a minimizing curve for the length functional of degree $2$, that means  $L_2(\alpha) \le L_2(\ga)$ for each curve $\ga:[a,b] \to P_v$,   $\ga(s)=(x(s),t(s))$ such that $\ga(a)=(x_0,t_0)$ and $\ga(b)=(x_1,t_1)$. Indeed, defining  the following  function 
\[
f:[a,b] \to \rr, \qquad f(s)=\escpr{t(s)-t_0, t_1- t_0},
\]
we have $f(a)=0 $ and $f(b)=|t_1- t_0|^2$. Then, it holds 
\begin{equation}
\label{eq:fb}
f(b)=f(b)-f(a)= \int_a^b f'(s) ds.
\end{equation}
By Cauchy-Schwarz inequality and \eqref{eq:fb} we obtain 
\begin{align*}
|t_1-t_0|^2= \left| \int_a^b f'(s) ds \right| \le \int_a^b |f'(s)| ds &\le  \int_a^b | \escpr{t'(s), t_1- t_0}| ds\\
 &\le |t_1-t_0| \int_a^b |t'(s)| ds.
\end{align*}
Then, it follows 
\begin{equation}
\label{eq:min}
|t_1-t_0| \le \int_a^b |t'(s)| ds=\int_a^b \sqrt{\escpr{\ga'(s),T}^2} ds=L_2(\ga).
\end{equation}
Since $L_2(\alpha_0)=0$ we have $L_2(\alpha)=L_2(\alpha_1)=|t_1-t_0|$. By equation \eqref{eq:min} we conclude $L_2(\alpha) \le L_2(\ga)$. However $L_2$ has several minimum among all curves that fix that end-points $p$ and $q$, because each curve of degree $2$ that has increasing $t$ coordinate is a minimum for the for the length functional $L_2$. Indeed, when we reach the horizontal leaf of coordinates $t_1$ we can connect each point on the leaf leaving unchanged the value of $L_2$.
\end{example}

\begin{example}[Characteristic plane]
Let $P_c$ be the characteristic (or horizontal) plane in $\mathbb{H}^1$ defined by
\[
P_c=\{(x,y,t) \in \mathbb{H}^1 \ : \ t=0\}.
\]
In cylindrical coordinate s$x=\rho \cos(\theta)$, $y=\rho \sin(\theta)$ and $t=t$, where $\rho>0$ and $\theta \in [0,2\pi]$, we consider the orthonormal adapted basis $(X',Y',T)$ in $\mathbb{H}^1$, where 
\begin{equation}
\begin{aligned}
X'&=\cos(\theta) X + \sin(\theta) Y= \dfrac{\partial }{\partial \rho},\\
Y'&=-\sin(\theta) X + \cos(\theta)Y=\dfrac{1}{\rho} \frac{\partial }{\partial \theta}+ \frac{\rho}{2} \frac{\partial}{\partial t}.\\
\end{aligned}
\end{equation}
Furthermore, since the tangent vector to $P_c$ are $\frac{\partial }{\partial \rho}$ and $$\frac{\partial }{\partial \theta}= \rho Y' - \frac{\rho^2}{2}T,$$ 
we deduce   
\[
Z=\frac{\partial }{\partial \rho} \quad \text{and} \quad S=\dfrac{1}{\rho \sqrt{1+\frac{\rho^2}{4}}} \frac{\partial }{\partial \theta}.
\]
Since $\escpr{[Z,S],S}=\frac{\partial }{\partial \rho}\left((1+\frac{\rho^2}{4})^{-\frac{1}{2}} \right) \ne0$ for all $\rho>0$ the geodesics for $L_2$ are integral curves of $S$. Let $\ga(t):[0,\bar{t}] \to \rr^2\setminus \{0\}$ be the integral closed curve of $S$ such that $\ga(0)=(\rho_0,\theta_0) $ and $\ga(\bar{t})=(\rho_0,\theta_0)$, $\rho_0>0$, $\theta_0 \in [0,2 \pi]$ and $|\cdot|$  the Euclidean metric. When $\rho_0$ tends to $0$ the circle described by $\ga$ collapses to the characteristic point $0$ and we have
\[
 \lim_{\rho_0 \to0} \int_{I} |\ga'(t)| dt =\lim_{\rho_0 \to 0} \dfrac{2 \pi \rho_0}{\rho_0 \sqrt{1+\frac{\rho_0^2}{4}}}= 2\pi.
\]
\end{example}

\begin{example}[Pansu's spheres]
In cylindrical coordinates the Pansu sphere $\mathbb{S}_1$ is the union of the graphs of the functions $f$ and $-f$ defined on the plane $t=0$, where for $0<\rho\le1$ 
\[
f(\rho,\theta)=\frac{1}{2} \left( \rho \sqrt{1-\rho^2}+ \cos^{-1}(\rho) \right).
\]
Then, for $0 < \rho <1$ we have 
\[
\frac{\partial f}{\partial \rho}=-\frac{\rho^2}{\sqrt{1-\rho^2}} \quad \text{and} \quad \frac{\partial f}{\partial \theta}=0.
\]
Therefore the unit normal $N$ to the upper (lower) hemisphere, described by the graph $f$ (respectively $-f$ ), is 
\[
N=\dfrac{1}{\sqrt{1-\rho^2}}\left(T\pm\frac{\rho^2}{\sqrt{1-\rho^2}} X'\right).
\]
Thus, we have $\nu_h=\pm X'$, $Z=\pm Y'$, $\escpr{N,T}=(1-\rho^2)^{-\frac{1}{2}}$ and 
\[
N_h=\pm \frac{\rho^2}{1-\rho^2} X' \quad \text{and} \quad S=\frac{1}{\sqrt{1-\rho^2}}X' \pm \frac{\rho^2}{1-\rho^2}T.
\]
A straightforward computation shows that for $0<\rho<1$ there holds 
\begin{equation}
\begin{aligned}
\escpr{[Z,S], S}&=\frac{1}{\sqrt{1-\rho^2} }\escpr{[Y' ,X'], S}\\
&=\frac{1}{1-\rho^2 }\escpr{\frac{1}{\rho}Y'+ T , X' +\frac{\rho^2}{\sqrt{1-\rho^2}}T}\\&
= \rho^2 (1-\rho^2)^{-\frac{3}{2}}\ne0.
\end{aligned}
\label{eq:c122}
\end{equation}
On the equator of $\mathbb{S}_1$, parametrized by $ \theta \to (1,\theta,0)$, we have $\escpr{[Z,S], S}=0$ since $T$ is tangent to $\mathbb{S}_1$. Fix a point $p_0=(\rho_0,\theta_0, \pm f(\rho_0,\theta_0)) \in \mathbb{S}_1 $ with $0<\rho<1$. By \eqref{eq:c122} out of the equator $\escpr{[Z,S], S}\ne0$, then the geodesic $\ga(s)=(\rho(s), \theta(s),  t(s)=f(\rho(s),\theta(s)))$ at $p_0$ for the $L_2$ functional is the solution of the following Cauchy problem 
\begin{equation}
\label{eq:geodpansu}
\begin{cases}
\dot{\rho}(s)= \frac{1}{\sqrt{1-\rho(s)^2}}\\
\dot{\theta}(s)=0\\
\dot{t}(s)=\mp \frac{\rho(s)^2}{1-\rho(s)^2}\\
\ga(0)=(\rho_0,\theta_0, \pm f(\rho_0,\theta_0)).
\end{cases}
\end{equation}
Then we have $ \theta(s)=\theta_0$ and $\rho(s)$ verifies 
\begin{equation}
\label{eq:pararc}
\frac{1}{2}\left[y \sqrt{1-y^2} + \sin^{-1}(y)\right]_{\rho_0}^{\rho(s)}=s.
\end{equation}
Moreover, we notice that the sign of $\dot{t}(s)$ is opposite to the sign of $f$. This means that the $t$ coordinate of a geodesic $\ga(s)$ decreases in the upper hemisphere and increases in lower hemisphere until $\ga$ reaches the equator. Then for each $\rho_0>0$ and $\theta_0 \in [0,2\pi]$ we consider $\ga_u:[0,\bar{s}] \to \mathbb{S}_1$ the solution of \eqref{eq:geodpansu}  such that  
$\ga_u(0)=(\rho, \theta, f(\rho))$, $\ga_u(\bar{s})=(1,\theta_0,0)$  and $\ga_{\ell}: [0,\bar{s}] \to \mathbb{S}_1$ the solution of \eqref{eq:geodpansu}  such that 
$\ga_u(0)=(\rho, \theta, -f(\rho))$, $\ga_{\ell}(\bar{s})=(1,\theta_0,0)$. Letting $I=[0,2\bar{s}]$ we set 
\[
\alpha(s)=
\begin{cases}
\ga_u (s) & s \in [0, \bar{s}] \\
\ga_{\ell}(2 \bar{s}- s) & s \in [\bar{s}, 2\bar{s}]
 \end{cases}
\]
When $\rho_0$ tends to $0$ the end-points of $\alpha$ go to the poles that are characteristic points and we have
\[
\lim_{\rho_0 \to0} \int_{I} |\alpha'(t)| dt =\lim_{\rho_0 \to 0}  2 \bar{s}= \frac{\pi}{2},
\]
since by \eqref{eq:pararc} it follows 
\[
\bar{s}=\lim_{\rho_0 \to 0} \lim_{\rho \to 1}\frac{1}{2}\left[y \sqrt{1-y^2} + \sin^{-1}(y)\right]_{\rho_0}^{\rho(s)}=\frac{\pi}{4}.
\]
\end{example}

\appendix

\section{Integrability of admissible vector fields on a regular curve}

\label{sec:integrability of admissible vector}
In this Appendix, we provide an alternative proof of the fundamental Theorem~3 in Hsu's paper \cite{MR1189496}, that implies that, when $\gamma$ is a regular curve in $(a,b)$, then any admissible vector field along $\gamma$ with compact support in $(a,b)$ is integrable. We need first some preliminary results.

We consider the following spaces
\begin{enumerate}
\item $\frak{X}^r_\ga(a)$, $r\ge 0$, is the set of $C^r$ vector fields along $\ga$ that vanish at $a$.
\item $\mathcal{H}_\ga^r(a)$, $r\ge 0$, is the set of horizontal $C^r$ vector fields along $\ga$ vanishing at $a$.
\item $\mathcal{V}_\ga^r(a)$, $r\ge 0$, is the set of vertical vector fields of class $C^r$ along $\ga$ vanishing at $a$. By a vertical vector we mean a vector in $\mathcal{H}^\perp$.
\end{enumerate}
We shall denote by $\Pi_v$ the orthogonal projection over the vertical subspace.

For $r\ge 1$, we consider the map
\begin{equation}
\label{eq:defG}
\mathcal{G}:\mathcal{H}_\ga^{r-1}(a)\times \mathcal{V}_\ga^r(a)\to \mathcal{H}_\ga^{r-1}(a)\times \mathcal{V}_\ga^{r-1}(a),
\end{equation}
defined by
\[
\mathcal{G}(Y_1,Y_2)=(Y_1,\mathcal{F}(Y_1+Y_2)),
\]
where $\mathcal{F}(Y)=\Pi_v(\Gamma(Y)')$, and $\Gamma(Y)(t)=\exp_{\gamma(t)}(Y(t))$. Observe that $\mathcal{F}(Y)=0$ if and only if the curve $\Gamma(Y)$ is horizontal.

We consider on each space the corresponding $||\cdot||_r$ or $||\cdot||_{r-1}$ norm, and the corresponding product norm (it does not matter whether it is Euclidean, the sup or the $1$ norm).

Then
\[
D\mathcal{G}(0,0)(Y_1,Y_2)=(Y_1,D\mathcal{F}(0)(Y_1+Y_2)),
\]
where $D\mathcal{F}(0)Y$ is given by
\[
D\mathcal{F}(0)Y=\sum_{i=k+1}^n \big(\escpr{\nabla_{\gamma'}Y,X_i}+\escpr{\gamma',\nabla_Y X_i}\big)\,X_i.
\] 
Oberve that $D\mathcal{F}(0)Y=0$ if and only if $Y$ is an admissible vector field.

Our objective now is to prove that the map $D\mathcal{G}(0,0)$ is an isomorphism of Banach spaces. To show this, we shall need the following result.

\begin{proposition}
The differential $D\mathcal{G}(0,0)$ is an isomorphism of Banach spaces.
\end{proposition}
\begin{proof}
We first observe that $D\mathcal{G}(0,0)$ is injective, since $D\mathcal{G}(0,0)(Y_1,Y_2)=(0,0)$ implies that $Y_1=0$ and that the vertical vector field $Y_2$ satisfies the compatibility equations with initial condition $Y_2(a)=0$. Hence $Y_2=0$. The map $D\mathcal{G}(0,0)$ is continuous.
Indeed, if for instance we consider the $1$-norm on the product space we have
\begin{align*}
 \| D\mathcal{G}(0,0)(Y_1, Y_2) \|&= \| (Y_1,D\mathcal{F}(0)(Y_1+Y_2)) \|\\
                        &\le \| Y_1 \|_{r-1} + \|D\mathcal{F}(0)(Y_1+Y_2)) \|_{r-1}\\
                        &\le (1+ \|(a_{ij}) \|_{r-1})\|Y_1\|_{r-1}+ (1+\|(b_{ij}) \|_{r-1})\|Y_2 \|_{r}.\\                        
\end{align*} 

To show that $D\mathcal{G}(0,0)$ is surjective, we take $(Y_1,Y_2)$ in the image, and we find a vector field $Y$ along $\gamma$ such that $Y(a)=0$, $Y_h=Y_1$ and $D\mathcal{F}(0)(Y)=Y_2$ by  Lemma~\ref{lem:existY_hY_v}. The map $D\mathcal{G}(0,0)$ is open because of the estimate \eqref{eq:DGopen} given in Lemma~\ref{lem:odeestimate} below.
\end{proof}

\begin{lemma}
\label{lem:odeestimate}
In the above conditions, assume that $D\mathcal{F}(0)(Y)=Y_2$ and $Y_h=Y_1$ and $Y(a)=0$. Then there exists a constant $K$ such that 
\begin{equation}
\label{eq:DGopen}
 \| Y_v \|_{r} \le K (\|  Y_2 \|_{r-1}+ \| Y_1  \|_{r-1})
\end{equation}
\end{lemma}

\begin{proof}
Reasoning as in Lemma \ref{lem:existY_hY_v}   we choose a global orthonormal adapted basis $(X_i)$ on $\ga$ and write 
\[
Y_1=\sum_{i=1}^k g_i X_i, \quad Y_2=\sum_{r=k+1}^n z_r X_r \quad \text{and} \quad Y_v=\sum_{r=k+1}^n f_r X_r.
\]
Then $Y_v$  is a solution of the ODE $\eqref{eq:compatibilitycond}$ given  by
\begin{equation}
\label{eq:ODEprop}
 F'=-B(t)F+ Z(t)-\ A(t) G(t)
\end{equation}
where $B(t), A(t)$ are defined in \eqref{def:AB}, $F$, $G$ are defined in \eqref{def:FG} and we set 
\[
Z=\begin{pmatrix} z_{k+1} \\ \vdots \\ z_n \end{pmatrix}.
\]
Since $Y_v(a)=0$ an $Y_v$ solves \eqref{eq:ODEprop} in $(a,b)$, by Lemma \ref{lm:ODE ineq} there exists a constant $K$ such that 
\begin{equation}
\begin{aligned}
 \| Y_v \|_{C^r([a,b])}=\| F \|_{C^r([a,b])} &\le K \|  Z(t)- A(t) \ G(t)  \|_{C^{r-1}([a,b])}\\
               &\le K' (\|  Y_2 \|_{C^{r-1}([a,b])}+ \| Y_1  \|_{C^{r-1}([a,b])}).
\end{aligned}
\end{equation}
where $K'=K\max\{1, \sup_{[a,b]}\| A(t)\|_{r-1}\}$.
\end{proof}

\begin{lemma}
\label{lm:ODE ineq}
Let $r\geqslant1$ be a natural number. Let $u:[a,b] \to \rr^d$ be the solution  of the inhomogeneous problem 
\begin{equation}
  \begin{cases}
  u'= A(t) u+ c(t),\\
  u(a)=u_0
  \end{cases}
  \label{eq:inhsystem}
\end{equation}
where $A(t)$ is a  $d\times d $ matrix in $C^{r-1}$ and $c(t)$ a $C^{r-1}$ vector field. Then,  there exists a constant $K$ such that 
 \begin{equation}
  \| u \|_r \leqslant K (\| c \|_{r-1}+ |u_0|).
  \label{dis:lemma ineq}
 \end{equation} 
\end{lemma}
\begin{proof}
The proof is by induction. We start from the case $r=1$. By \cite[Lemma 4.1] {Hartman} it follows 
\[
 u(t)\leqslant \left(|u_0|+ \int_{a}^{t} |c(s)| ds \right) e^{|\int_{a}^{t} \|A(s) \| ds |},
\]
where the norm of $A$ is given by $\sup_{|x|=1}|A \ x|$.
Therefore we have 
\begin{equation}
\label{dis:u}
 \sup_{t \in [a,b]} |u(t)| \le C_1  (\sup_{t \in [a,b]} |c(t)|+|u_0|) ,
\end{equation}
where we set
\[
 C_1=(b-a) e^{(b-a) \sup_{t \in [a,b]} \|A(t) \| }.
\]
Since $u$ is a solution of \eqref{eq:inhsystem} it follows 
\begin{equation}
\begin{aligned}
 \sup_{t \in [a,b]} |u'(t)|&\le \sup_{t \in [a,b]}  \|A(t) \| \  \sup_{t \in [a,b]} |u(t)|+  \sup_{t \in [a,b]} |c(t)|\\
 &\le ( C_2 + 1) \sup_{t \in [a,b]} |c(t)|.\\
\end{aligned}
\label{dis:u'}
\end{equation}
Hence by \eqref{dis:u} and \eqref{dis:u'} we obtain 
\[
 \| u \|_1 \le K (\| c \|_0 +|u_0|).
\]
Assume that \eqref{dis:lemma ineq} holds for $1\le k\le r$, then by \eqref{eq:inhsystem} we have
\[
 u^{(r+1)}(t)=\sum_{k=0}^r A^{(k)}(t)  \ u^{(r-k)}(t)+ c^{(r)}(t).
\]
By the inductive assumption we deduce
\begin{equation}
\begin{aligned}
 \sup_{t \in [a,b]} |u^{(r+1)}(t)|&\le  \sum_{k=0}^r  \sup_{t \in [a,b]}  \|A^{(k)}(t) \| \  \sup_{t \in [a,b]} | u^{(r-k)}(t)|+  \sup_{t \in [a,b]} |c^{(r)}(t)|\\
 &\le ( C_3 + 1) ( \sup_{t \in [a,b]} |c^{(r)}(t)| +|u_0|).\\
\end{aligned}
\label{dis:ur}
\end{equation}
Hence the inequality \eqref{dis:lemma ineq} for $r+1$ simply follows by \eqref{dis:ur}.
\end{proof}

Finally, we use the previous constructions to give a criterion for the integrability of admissible vector fields along a horizontal curve.

\begin{theorem}
\label{thm:intcriterion}
Let $\ga:I\to N$ be a horizontal curve in a Carnot manifold $(N,\mathcal{H})$ endowed with a Riemannian metric. Assume that $\ga$ is regular in the interval $[a,b]\subset I$. Then every admissible vector field with compact support in $(a,b)$ is integrable.
\end{theorem}

\begin{proof}
Let us take $V,V^1,\ldots,V^{n-k}$ vector fields along $\ga$ vanishing at $a$. We consider the map
\[
\tilde{\mathcal{G}}:\big[(-\eps,\eps)\times (-\eps,\eps)^{n-k}\big]\times \big[\mathcal{H}_\ga^{r-1}(a)\times \mathcal{V}_\ga^r(a)\big]\to \mathcal{H}_\ga^{r-1}(a)\times \mathcal{V}_\ga^{r-1}(a),
\]
given by
\[
\tilde{\mathcal{G}}((s,(s_i),Y_1,Y_2))=(Y_1,F(sV+\sum_{i=1}^{n-k}s_iV^i+Y_1+Y_2)).
\]
The map $\tilde{\mathcal{G}}$ is continuous with respect to the product norms (on each factor we put the natural norm, the Euclidean one on the intervals and $||\cdot||_r$ and $||\cdot||_{r-1}$ in the spaces of vectors along $\ga$). Moreover
\[
\tilde{\mathcal{G}}(0,0,0,0)=(0,0),
\]
since the curve $\ga$ is horizontal. Now we have that
\[
D_2\tilde{\mathcal{G}}(0,0,0,0)(Y_1,Y_2)=D\mathcal{G}(0,0)(Y_1,Y_2)
\]
is a linear isomorphism. We can apply the Implicit Function Theorem to obtain maps
\[
Y_1:(-\eps,\eps)^{n-k+1}\to \mathcal{H}_\ga^{r-1}(a), \quad Y_2:(-\eps,\eps)^{n-k+1}\to \mathcal{V}_\ga^r(a),
\]
such that $\tilde{\mathcal{G}}(s,(s_i),(Y_1)(s,s_i),(Y_2)(s,s_i))=(0,0)$. This implies that $(Y_1)(s,(s_i))=0$ and that
\[
F(sV+\sum_i s_iV^i+Y_2(s,s_i))=0.
\]
Hence the curves
\[
\Gamma(sV+\sum_i s_iV^i+Y_2(s,s_i))
\]
are horizontal.

Now we assume that $V$ is an admissible vector field with $V(a)=V(b)=0$, and that $V^1,\ldots,V^{n-k}$ are admissible vector fields vanishing at $a$. Then the vector field
\[
\frac{\ptl Y_1}{\ptl s}(0,0),\frac{\ptl Y_1}{\ptl s_i}(0,0)
\]
along $\gamma$ are vertical and admissible. Since they vanish at $a$, they are identically $0$.

If, in addition, $V^1_v(b),\ldots,V^{n-k}_v(b)$ generate the space $\mathcal{H}^\perp(b)$. We consider the map
\[
\Pi:(-\eps,\eps)^{n-k+1}\to N
\]
given by
\[
(s,(s_i))\mapsto \Gamma(sV+\sum_i s_iV^i+Y_2(s,s_i))(b).
\]
For $s$, $(s_i)$ small, the image of this map is an $(n-k)$-dimensional submanifold $S$ of $N$ with tangent space at $\ga(b)$ given by $\mathcal{H}^\perp_{\ga(b)}$ (as $V(b)=0$ and $V^i(b)=V^i_v(b)$ generate $\mathcal{H}^\perp_{\ga(b)}$).
Notice that 
\[
 \dfrac{\partial \Pi (0,0)}{\partial s_i}=V^i(b)=V_v^i(b),
\]
which is invertible and 
\[
 \dfrac{\partial \Pi (0,0)}{\partial s}=V(b)=0.
\]
Hence we can apply the Implicit Function Theorem to conclude that there exists a family of smooth functions $s_i(s)$ so that
\[
\Gamma(sV+\sum_i s_i(s) V^i+Y_2(s,s_i(s)))
\]
are horizontal and take the value $\ga(b)$ at $b$. Clearly, we have
\[
 \Pi(s,(s_i(s)))=\gamma(b).
\]
Differentiating with respect to $s$ at $s=0$ we obtain 
\[
\dfrac{\partial \Pi (0,0)}{\partial s} + \sum_i \dfrac{\partial \Pi (0,0)}{\partial s_i} s_i'(0)=0.
\]
Therefore $s_i'(0)=0$ for each $i=1, \ldots, n-k$. Thus, the variational vector field to $\Gamma$ is 
\begin{equation}
\dfrac{\Gamma(s)}{\partial s}\bigg|_{s=0}=V+\sum_{i} s_i'(0) V^i+ \frac{\ptl Y_2}{\ptl s}(0,0)+\sum_{i} \frac{\ptl Y_2}{\ptl s_i}(0,0)=V.\qedhere
\end{equation}
\end{proof}

\section{The holonomy map on the space of square integrable functions}
\label{ap:holonomyL2}
In \cite[Section 3.8]{Montgomery} R. Montgomery stressed the fact the $C^1$ topology is not the correct one for calculus of variations. Indeed, a rigid curve in the $C^1$ topology is always a local minimum (see \cite{Sussman}) for each functional, in particular for the length functional, since its minimality does not depend on the functional but only on its domain. Therefore R. Montgomery suggested to consider the $W^{1,2}$ topology for curves instead of the $C^1$ topology, introducing the endpoint map described in  \cite[Chapter 5]{Montgomery}. Here we show how we can take into account this weakening of the regularity for the holonomy map.

Let $I \subset \rr$ be an open interval. Let  $\ga:I\to N$ be an absolutely continuous curve of degree $d=\deg(\ga)$ with square integrable derivative
\[
\ga'(t)=\sum_{\ell=1}^k u_{\ell} (t) (X_{\ell})_{\ga(t)},
\]
where $u_{\ell} \in L_{loc}^2(I,\rr)$ for each $\ell=1,\ldots,k$ letting $\mh:=\mh^d$, $k=n_d$, $((X_1)_{\ga(t)},$ $ \ldots, (X_k)_{\ga(t)})$ is an horizontal frame along $\ga$  for $\mh$ and $((X_{k+1})_{\ga(t)}, \ldots, (X_n)_{\ga(t)})$ is a vertical frame for $\mathcal{V}=(\mh)^{\perp}$  along $\ga$, both of them provided by Remark \ref{rk:horizontalconnection}. Let  $[a,b]\subset I$. A square integrable vector field $V \in L^2([a,b], TN)$ can be projected into its horizontal part $V_h \in L^2([a,b], \mh)$ given by 
\[
V_h=\sum_{i=1}^k g_i(t) (X_i)_{\ga(t)} \quad \text{where} \quad g_i \in L^2([a,b])
\]
and its vertical part $V_v \in L^2([a,b], \mathcal{V}) $ given by
\[
V_v=\sum_{r=k+1}^n f_r(t) (X_r)_{\ga(t)} \quad \text{where} \quad f_r \in L^2([a,b]).
\]
Thus, the admissibility system  \eqref{eq:compatibilitycond} is now equivalent to 
\begin{equation}
\label{eq:distcompatibilitycond}
f_r'+\sum_{i=1}^k a_{ri}g_i+\sum_{j=k+1}^n b_{rj} f_j= 0,\quad r=k+1,\ldots,n,
\end{equation}
where we consider distributional derivatives and the coefficients 
\begin{equation}
\label{def:abl2}
a_{ri}=\sum_{\ell=1}^k u_{\ell}(t) c_{\ell i}^r( \ga), \quad \quad b_{rj}=\sum_{\ell=1}^k u_{\ell}(t) c_{\ell j}^r(\ga)
\end{equation}
belong to $L^2([a,b])$, since by assumption $u_{\ell} \in L^2([a,b])$.
Since any linear system satisfies the Carathéodory hypothesis \cite[eq. (5.2), Section I.5]{Hale80} and  the Lipschitz condition in the spatial variable \cite[eq. (5.3), Section I.5]{Hale80}, we have that the system \eqref{eq:distcompatibilitycond} admits a unique absolutely continuous solution by \cite[Theorem 5.3, Section I.5]{Hale80}.
Hence this Carathéodory's existence theorem allows us to define a holonomy  type map
\[
\tilde{H}_\gamma^{a,b}: L^2([a,b], \mh)\to \mathcal{V}_{\ga(b)}
\]
where $\mathcal{V}_{\ga(b)}$ is the vector space of vertical vectors at the point $\ga(b)$. In order to define $\tilde{H}_\gamma^{a,b}$ we consider a horizontal vector $V_h\in L^2([a,b], \mh)$ and we take the only vector field $V_v\in W^{1,2}([a,b], \mathcal{V})$ solution of \eqref{eq:distcompatibilitycond} with initial condition $V_v(a)=0$, thanks to \cite[Theorem 5.3, Section I.5]{Hale80}.  By the Sobolev Embedding Theorem \cite[Corollary 7.11]{GT01}  the space $W^{1,2}([a,b], \mathcal{V})$ is continuously embedded in $C^{\frac{1}{2}}([a,b], \mathcal{V})$. Thus, we consider the $\frac{1}{2}$-H\"older function, denoting it always as $V_v$, in the class of functions of $V_v \in W^{1,2}([a,b], \mathcal{V})$   so that  we define
\[
\tilde{H}_\gamma^{a,b}(V_h)=V_v(b).
\]
\begin{definition}
In the above conditions, we say that $\gamma$ restricted to $[a,b]$ is \emph{regular} if the holonomy map $\tilde{H}_\gamma^{a,b}$ is surjective.
\end{definition}

Defining $F$ and $G$ as in \eqref{def:FG}, the system \eqref{eq:distcompatibilitycond} is equivalent to $F'=-BF-AG$, where $A$, $B$ are the  $L^2$ matrices defined in \eqref{def:abl2} and the time derivative shall be understood in the distributional sense. In these conditions, the coordinates of $\tilde{H}_\ga^{a,b}(V_h)=V_v(b)$ in the basis  $(X_i)$ are given by $F(b)$. 

The following result allows the integration of the differential equation \eqref{eq:distcompatibilitycond} to explicitly compute the holonomy map. 

\begin{proposition}
\label{prop:intholl2}
In the above conditions, there exists a square regular matrix $D(t)$ of order $(n-k)$ with coefficient in $C^{\frac{1}{2}}([a,b])$ such that
\begin{equation}
\label{eq:F(b)l2}
F(b)=- D(b)^{-1}\int_a^b (D A)(t) G(t) \, dt.
\end{equation}
\end{proposition}

\begin{proof}
Lemma~\ref{lm:detl2} below allows us to find a regular matrix $D(t)$ with coefficient in $C^{\frac{1}{2}}([a,b])$ such that $D'= DB$. Then equation $F'=-BF-AG$ is equivalent to $(DF)'=- DAG$. Since $DF$ belongs to $W^{1,2}([a,b], \mathcal{V})$ and the fundamental theorem of calculus still holds in $W^{1,2}$  we have 
\[
D(b)F(b)-D(a)F(a)=-\int_a^b (D A)(t) G(t) \, dt
\]
Taking into account that $F(a)=0$, and multiplying by $D(b)^{-1}$, we obtain \eqref{eq:F(b)l2}.
\end{proof}

\begin{lemma}
\label{lm:detl2}
Let $B(t)$ be a $L^2$ family of square matrices on the interval $[a,b]$. Let $D(t)$ be the $C^{\frac{1}{2}}$ solution of the Cauchy problem
\begin{equation}
\label{eq:omol2}
D'(t)=D(t) B(t)\ \text{on }[a,b],  \quad D(a)=I_d.
\end{equation}
Then $\det  D(t)\ne0$ for each $t \in [a,b]$.
\end{lemma}
\begin{proof}
Thanks to \cite[Theorem 5.3, Section I.5]{Hale80} the Cauchy problem \eqref{eq:omol2} has a unique solution $D \in W^{1,2}([a,b],\mathbb{M}^{d\times d})$, that belongs to $C^{\frac{1}{2}}([a,b],\mathbb{M}^{d\times d})$ by the Sobolev Embedding Theorem. 
Since the determinant is a polynomial function, therefore $C^1$, we  apply the chain rule in Sobolev spaces to gain  the Jacobi formula
\[
\dfrac{d (\det D(t))}{dt}=\text{Tr}\left(\text{adj}\, D(t)\, \dfrac{d D(t)}{dt}\right)
\]
in the distributional sense, where $\text{adj} D$ is the classical adjoint (the transpose of the cofactor matrix) of $D$ and $\text{Tr}$ is the trace operator. Therefore
\begin{equation}
\label{eq:detode}
\dfrac{d \det(D(t))}{dt}=\text{Tr}\left((\text{adj}\, D(t)) D(t) B(t) \right)=\det D(t)\, \text{Tr}(B(t)).
\end{equation}
Since $\det D(a)=1$, the solution for \eqref{eq:detode} is given by 
\[
 \det D(t)=e^{\int_a^t \text{Tr}(B(\tau)) \, d \tau}>0,
\]
for all $t\in [a,b]$. Thus, the matrix $D(t)$ is invertible for each $t \in [a,b]$.
\end{proof}

\begin{theorem}
\label{th:singcharcl2}
The absolutely continuous curve $\ga$ of degree $d=\deg(\ga)$, with square integrable derivative, is non-regular restricted to $[a,b]$  if and only if there exists a $C^{\frac{1}{2}}$ row vector field $\Lambda(t)\ne0$  for all $t \in [a,b]$  that solves the following system
\begin{equation}
 \begin{cases}
 \Lambda'(t)= \Lambda(t) B(t)\\
 \Lambda(t) A(t)=0,
 \end{cases}
 \label{eq:singularsyscl2}
\end{equation}
for a.e. $t \in [a,b]$.
\end{theorem}
\begin{proof}
 
Assume that $\ga$ is nonregular in $[a,b]$, then the image of the holonomy map is contained in a proper subspace of $\mathcal{V}_{\ga(b)}$. Therefore there exists a row vector $\Gamma \ne 0$ such that 
\begin{equation}
\label{eq:gddag}
 \Gamma F(b)=- \int_a^b \Gamma D(b)^{-1} D(t) A(t) G(t)=0
\end{equation}
for all $G\in L^2([a,b],\mh)$, where $D(t)$ solves 
\begin{equation}
\label{eq:homDcl2}
\begin{cases}
D(t)'= D(t) B(t)\\
D(a)=I_{n-k}.
\end{cases}
\end{equation}
In the previous computation we used the integral formula provided by Proposition~\ref{prop:intholl2}.
Setting $\Lambda(t):= \Gamma D(b)^{-1} D(t)$ by equation \eqref{eq:gddag} we obtain $\Lambda(t)A(t)=0$ for a.e. $t \in [a,b]$. 
Since $\Gamma$ is a constant vector and $D(t)$ is a regular  $C^{\frac{1}{2}}([a,b])\cap W^{1,2}([a,b])$ matrix by Lemma~\ref{lm:detl2},  we obtain $ \Lambda'(t)=\Lambda(t) B(t)$ a.e. in $[a,b]$ and  $\Lambda(t) \ne 0$ for all $t \in [a,b]$.

Conversely, any solution of the system \eqref{eq:singularsyscl2} is given by 
\[
\Lambda(t)= \Gamma D(t),
\]
where $\Gamma=\Lambda(0)\ne 0$ and $D(t)$ solves the equation \eqref{eq:homDcl2}.
Indeed, let us consider a general solution $\Lambda(t)$ of \eqref{eq:singularsyscl2}. If we set
\[
\Phi(t)=\Lambda(t)-\Gamma D(t),
\]
where $\Gamma=\Lambda(0)\ne 0$ and $D(t)$ solves the equation \eqref{eq:homDcl2}, then we deduce 
\[
\begin{cases}
\Phi(t)'=\Phi(t)B(t)\\
\Phi(0)=0.
\end{cases}
\]
Clearly the unique solution of this system is $\Phi(t)\equiv 0$.
Hence we  conclude that $\Gamma D(t) A(t)=0$ for a.e. $t$ in $[a,b]$. Furthermore by Proposition~\ref{prop:intholl2} we have that the image of the holonomy map is given by 
\[
F(b)=- D(b)^{-1}\int_a^b (D A)(t) G(t) \, dt
\]
for each $G \in L^2([a,b],\mh)$. Setting $\tilde{\Gamma}:=\Gamma D(b)$ we obtain 
\[
\tilde{\Gamma} F(b)=- \int_a^b \Gamma D(t) A(t) G(t) dt=0.
\]
Therefore the image of the holonomy map is contained in a proper subspace of $\mathcal{V}_{\ga(b)}$, thus the curve $\ga$ is non-regular restricted to $[a,b]$.
\end{proof}

\begin{remark} We notice that 
\begin{itemize}
\item as we pointed out in Remark~\ref{rk:Mchequalsing} when $(N,\mh)$ is a Carnot manifold the system  \eqref{eq:singularsyscl2} coincides with the characteristic system in  \cite[Lemma 5.2.3]{Montgomery} , but now for  absolutely continuous curves with square integrable derivatives and not only for $C^1$ curves;
\item we have the following inclusions for the holonomy map 
\[
\begin{tabular}{cccc}
$\tilde{H}_{\ga}^{a,b}:$ & $L^2([a,b],\mh)$& $\rightarrow$& $\mathcal{V}_{\ga(b)}$\\
&\rotatebox{90}{$\subset$}& &\\
$H_{\ga}^{a,b}: $ & $C_0((a,b),\mh)$ & $\rightarrow$ &$ \mathcal{V}_{\ga(b)}$\\
&\rotatebox{90}{$\subset$}& &\\
$H_{\ga}^{a,b}: $ & $C_0^1((a,b),\mh)$ & $\rightarrow$ &$ \mathcal{V}_{\ga(b)}$\\
&\rotatebox{90}{$\subset$}& &\\
 & $\vdots$ & & $\vdots$\\
 $H_{\ga}^{a,b}: $ & $C_0^r((a,b),\mh)$ & $\rightarrow$ &$ \mathcal{V}_{\ga(b)}$\\
 &\rotatebox{90}{$\subset$}& &\\
 & $\vdots$ & & $\vdots$\\
 $H_{\ga}^{a,b}: $ & $C_0^{\infty}((a,b),\mh)$ & $\rightarrow$ &$ \mathcal{V}_{\ga(b)}$,\\
\end{tabular}
\]
where the suitable control space depends on the regularity $(L^2, C^1, \ldots, C^{\infty} )$ of the immersed curve we consider. When the curve is $W^{1,2}$ the control space is $L^2([a,b],\mh)$ and when the curve is $C^{r}$ the control space for the holonomy map is $C_0^{r-1}((a,b),\mh)$ for $r \ge1$.
\end{itemize}
\end{remark}

\bibliography{degree}

\end{document}